\documentclass[11pt,reqno,oneside]{amsart}

\usepackage[utf8]{inputenc}
\usepackage[T1]{fontenc}
\usepackage[top=3cm,bottom=3cm,left=3cm,right=3cm]{geometry}

\usepackage{hyperref}
\usepackage{mathtools,amsmath}
\usepackage{thmtools}
\usepackage[capitalize, nameinlink]{cleveref}
\usepackage{graphicx}

\usepackage{xcolor}
\definecolor{red}{rgb}{1,0,0}
\hypersetup{
    colorlinks,
    linkcolor={red!50!black},
    citecolor={blue!50!black},
    urlcolor={blue!80!black}
}

\usepackage{amsfonts}
\usepackage{amssymb}
\usepackage{latexsym}
\usepackage{mathrsfs}
\usepackage{textcomp}
\usepackage{longtable}
\usepackage{booktabs}
\usepackage{graphicx}
\usepackage{setspace}
\usepackage{nicefrac}
\usepackage{enumerate}
\usepackage{wasysym}
\usepackage[normalem]{ulem}
\usepackage{upgreek}
\usepackage{paralist}
\usepackage{etoolbox}
\usepackage{mathdots}
\usepackage{wrapfig}
\usepackage{floatflt}
\usepackage{tensor} 
\usepackage{lscape}
\usepackage{stmaryrd}

\usepackage{tikz}
\usetikzlibrary{arrows}
\usetikzlibrary{matrix}
\usetikzlibrary{positioning}
\usetikzlibrary{calc}

\usepackage[all,cmtip]{xy} 
\usepackage[labelfont=bf, font={small}]{caption}[2005/07/16]

\newtheorem{thm}{Theorem}[section]

\usepackage{algorithm}
\usepackage{algpseudocode}

\newtheorem{lm}[thm]{Lemma}
\newtheorem{prop}[thm]{Proposition}
\newtheorem{cor}[thm]{Corollary}
\theoremstyle{definition}

\newtheorem{re}[thm]{Remark}
\newtheorem{ex}[thm]{Example}

\crefname{thm}{Theorem}{Theorems}
\crefname{lm}{Lemma}{Lemmas}
\crefname{prop}{Proposition}{Propositions}
\crefname{cor}{Corollary}{Corollaries}
\crefname{definition}{Definition}{Definitions}
\crefname{re}{Remark}{Remarks}
\crefname{ex}{Example}{Examples}
\crefname{algorithm}{Algorithm}{Algorithms}

\numberwithin{equation}{section}

\newcommand{\CC}{{\mathbb C}}
\newcommand{\RR}{{\mathbb R}}
\newcommand{\QQ}{{\mathbb Q}}
\newcommand{\ZZ}{{\mathbb Z}}

\newcommand{\PP}{{\mathbb P}}

\newcommand{\im}{\operatorname{im}}

	{\begin{figure} \begin{center}}%
	{\end{center} \end{figure}}

\newcommand{\id}{\operatorname{id}}

\newcommand{\rk}{\operatorname{rk}}

\newcommand{\diag}{\operatorname{diag}\nolimits}

\newcommand{\GL}{\operatorname{GL}\nolimits}
\newcommand{\Hom}{\operatorname{Hom}\nolimits}
\newcommand{\Id}{\operatorname{Id}\nolimits}

\newcommand{\ad}{\operatorname{ad}\nolimits}

\newcommand{\liea}[1]{\mathfrak{#1}}

\newcommand{\dto}{\dashrightarrow}

\renewcommand{\phi}{\varphi}

\newcommand{\bbQ}{{\mathbb{Q}}}
\newcommand{\calL}{\mathcal{L}}

\newcommand{\frakS}{{\mathfrak{S}}}

\newcommand{\frakg}{{\mathfrak{g}}}
\newcommand{\fraksl}{{\mathfrak{sl}}}
\newcommand{\frakgl}{{\mathfrak{gl}}}

\newcommand{\vvirg}{ ,\ldots, }
\newcommand{\ttimes}{ \times \cdots \times }
\newcommand{\ootimes}{ \otimes \cdots \otimes }

\title[Computing the symmetry algebra]{Computing the continuous symmetries\\ of a parametrized variety}

\author[B. Biaggi]{Benjamin Biaggi}
\address{Benjamin Biaggi, Mathematical Institute, University of Bern, Alpeneggstrasse
22, 3012 Bern, Switzerland}
\email{benjamin.biaggi@unibe.ch}

\author[J. Draisma]{Jan Draisma}
\address{Jan Draisma, Mathematical Institute, University of Bern, Sidlerstrasse 5,
3012 Bern, Switzerland}
\email{jan.draisma@unibe.ch}

\author[F. Gesmundo]{Fulvio Gesmundo}
\address{Fulvio Gesmundo, Institut de Mathématiques de Toulouse; UMR5219 -- Université de Toulouse; CNRS -- UPS, F-31062 Toulouse Cedex 9, France}
\email{fgesmund@math.univ-toulouse.fr}

\author[A. Maraj]{Aida Maraj}
\address{Aida Maraj, Max Planck Institute of Molecular Cell Biology and Genetics, and  Center for Systems Biology Dresden}
\email{maraj@mpi-cbg.de}

\author[M. Mi\v{s}inov\'a]{Magdal\'ena Mi\v{s}inov\'a}
\address{Magdal\'ena Mi\v{s}inov\'a, Department of Mathematics and Statistics, University of Konstanz, Universitätsstrasse 10, 78464 Konstanz, Germany}
\email{magdalena.misinova@gmail.com}

\thanks{BB is funded by JD's Swiss National Science Foundation project
grant 200021-227864.}


\begin{document}

\begin{abstract}
We prove that the symmetry Lie algebra of a parametrized variety can be determined directly from the parametrization, without computing the vanishing ideal of the variety. We derive a practical polynomial-time Monte Carlo algorithm for computing the symmetry Lie algebra of a parametrized variety. {{We discuss applications to testing the binomiality of the ideal of a parametrized variety after changing coordinates, and test this property on varieties arising from staged tree models and colored Gaussian graphical models. Finally, we discuss symmetries and binomiality after changing coordinates for rational curves and give a characterization of the symmetries of many secant varieties. }}
\end{abstract}

\maketitle
\setcounter{tocdepth}{1}

\section{Introduction}
In applications of algebraic geometry, one is often given a polynomial map or, more generally, a rational map $\phi\colon \CC^m \dto \CC^n$, and wants to compute properties of the image closure
\[
 X \coloneq \overline{\{\phi(p) \mid p \in \CC^m \text{ and } \phi 
\text{ is defined at } p\}.} 
\]
The closure and other topological notions throughout the paper 
refer to the Zariski topology; in this setting, the closure coincides with that in the Euclidean topology. This framework yields a large class of
(irreducible) algebraic varieties, called \emph{unirational varieties}.

{{Let $\GL_n \coloneq \GL_n(\CC)$ and $\frakgl_n\coloneq \frakgl_{n}(\CC)$  be the general linear group and its Lie algebra, respectively.}} An important invariant of the embedded variety $X$ is the group
\[
 G_X  \coloneq \{g \in \GL_n \mid gX = X \} 
 \]
of linear symmetries of $X$, often called the \emph{symmetry group} or the
\emph{linear preserver subgroup} of $X$. Knowledge of
$G_X$ can be exploited in the computation of other properties of $X$. For
instance, the vanishing ideal of $X$ in the polynomial ring $\CC[y_1 ,\ldots, y_n]$
 is a $G_X$-module, and thus knowledge of
the representation theory of $G_X$ may help in determining (interesting subspaces of) that vanishing ideal.

By construction, $G_X$ is a Zariski closed subgroup of $\GL_n$, and hence
the union of finitely many connected components, which by standard
algebraic group theory are also irreducible. Let $G_X^0$ be the
connected component containing the identity element ${{\mathrm{Id}_n}} \in \GL_n$. This
is a closed normal subgroup of $G_X$, and $G_X/G_X^0$ is a finite group;
see, e.g., \cite[\S1.2]{Bor91}. The identity component $G_X^0$ is uniquely
determined by the Lie algebra $\frakg_X$ of $G_X$. We call $\frakg_X$
the \emph{symmetry Lie algebra} of $X$. 

The problem of determining the symmetry group of an algebraic variety appears in several contexts \cite{GUTERMAN200061,LiPierce}, in particular in the setting of varieties of matrices, tensors, or operators with special structure \cite{LiTsing,tan2003,johnston2011}. The symmetry Lie group of varieties of matrices of bounded rank was determined in \cite{GUTERMAN200061}, and \cite{Westwick} determined the  {symmetry group} of the variety of rank-one tensors. More recently, \cite{GHL} generalized this result to a wide range of secant varieties of tensors and other related varieties. In \cite{Maraj23}, the symmetry Lie algebra of an algebraic variety $X$ is computed from the polynomials vanishing on $X$, and in \cite{Kahle25,Maraj23}, $\frakg_X$ is used to test whether the vanishing ideal of $X$ is binomial after an invertible linear change of variables. 

The purpose of this paper is to show that the Lie algebra $\frakg_X$ of $G_X$ can be computed directly from the parametrization $\phi$ without first computing the defining ideal of $X$. We give practical polynomial-time Monte Carlo algorithms for computing $\frakg_X$ and detecting binomiality after an invertible linear change of variables. We show that this insight is useful to derive theoretical results as well: we apply it to varieties of staged tree models and colored Gaussian graphical models, to rational curves and to secant varieties.

The following paragraphs briefly highlight the main contributions of this paper.

\subsection*{Computing the symmetry Lie algebra}
{{We characterize the symmetry Lie algebra of the parametrized variety $X$
as the algebra of linear maps that send each point of $X$ into its tangent
space at that point (\Cref{thm: frakg via tangent spaces}), and consequently, into the span of the columns of the Jacobian of $\phi$ (\Cref{cor: symmetry via jacobian}). For dimensionality reasons, only a finite number of points of $X$ are needed for the computation. We implement this in the probabilistic \Cref{alg:Symmetry} and the deterministic \Cref{alg:Symmetry2}.}}

In all our algorithms we assume that $\phi$ is defined over the field
of rational numbers $\bbQ$; in this case, $\frakg_X$ is the $\CC$-linear
span of $\frakg_{X,\QQ} \coloneq \frakg_X \cap \frakgl_n(\QQ)$, and our algorithm
computes a basis for $\frakg_{X,\QQ}$. This restriction on the coefficients
is very mild: on the one hand, in many applications $\phi$ is indeed
defined over $\QQ$; and on the other hand, the algorithms can easily
be extended to computations in a fixed finite-degree extension of~$\QQ$.
 They can also be extended to transcendental extensions, but they may no longer run in polynomial time, as rational-function coefficients can become super-polynomial in size.

\begin{thm}[{{\Cref{alg:Symmetry} and \Cref{prop:Symmetry}}}] \label{thm:Main}
For any $\epsilon>0$ there exists a polynomial-time Monte Carlo algorithm that, on input a list of rational functions $\phi=(f_1 ,\ldots, f_n)\in \QQ(x_1 ,\ldots, x_m)^n$, outputs a basis for $\frakg_{X,\QQ}$ with probability $\geq 1-\epsilon$. Here $X \subseteq \CC^n$ is the image closure of $\phi$ regarded as a rational map $\CC^m \dto \CC^n$. With probability $\leq \epsilon$, the output is either ``fail'' or a basis of a subspace of $\frakgl_n(\QQ)$ different from $\frakg_{X,\QQ}$.
\end{thm}

The input $\phi$ can be given in either of two ways: 
\begin{enumerate}
\item the \emph{standard encoding}, where each $f_i$ is given as an explicit rational function $a_i/b_i$, with $a_i$ and $b_i$ polynomials in $\QQ[x_1 ,\ldots, x_m]$ represented as a list of pairs $(c_\alpha,\alpha)$ with $\alpha \in \ZZ_{\geq 0}^m$ representing the exponent vector of the monomial $x^\alpha$ and $c_\alpha \in \QQ$ recording its coefficient;

\item the \emph{circuit encoding}, where each $f_i$ is given as an algebraic circuit in the sense of \cite{SY10}: nodes with in-degree $0$ are labelled by a constant or a variable $x_i$, nodes with in-degree $1$ compute the inversion $a \mapsto a^{-1}$, and nodes with in-degree $2$ compute addition or multiplication. We assume that the inversion gates never receive as input a function that is identically $0$. Such an algebraic circuit would not define any rational function. 
 \end{enumerate}
The
running time  of \Cref{alg:Symmetry} is polynomial in the bit length of the input
$\phi$, given in either of the two encodings, and in
$\log(1/\epsilon)$ (\Cref{prop:running time for Algorithm 1}).

The result extends to more powerful models of computation. In the
BSS model \cite{BSS89} extended with, say, sampling random real numbers from
the uniform distribution on $[0,1]$, one may allow $\phi$ to be defined
over $\RR$ or even $\CC$. In this case, \Cref{alg:Symmetry} runs in
polynomial time in the size of the input and returns a basis for $\frakg_{X}$ with
probability $1$.  

{{For a deterministic approach, we give \Cref{alg:Symmetry2}, obtained by derandomizing \Cref{alg:Symmetry}. In this case, we {{lose}} polynomiality of the running time as rational function expressions can potentially become large.}}

\subsection*{Testing $\GL$-binomiality}

In \cite{Maraj23,Kahle25}, the problem of determining the symmetry Lie algebra is studied within the context of toric varieties, particularly to establish necessary conditions for a variety $X \subseteq \mathbb{C}^n $ to have a \emph{$\GL$-binomial ideal}---that is,  for there to exist matrix $g\in \GL_n$ such that~$gX$ is the vanishing set of an ideal generated by polynomials with at most two terms. {{The condition $\dim \frakg_X\geq \dim X$ is presented as a necessary one in \cite{Maraj23}.}} This result is generalized in \cite{Kahle25}, which presents an algorithm for deciding whether $X$ has a binomial ideal after a coordinate change. {This algorithm requires access to the defining ideal $I(X)$: first it computes the symmetry Lie algebra $\mathfrak{g}_X$, then identifies the Lie algebra of a maximal torus $ T $ within $ G_X $, and finally verifies whether $ T $ has a dense orbit in $ X $. We discuss a Monte Carlo variant of this approach which does not require access to the defining ideal~$I(X)$ and it is applicable directly on a parametrization of $X$.}

\begin{cor} [{{\Cref{alg:Binom} and \Cref{prop: alg2 probability}}}] \label{cor:Binomial}
For any $\epsilon > 0$ there exists a polynomial-time Monte Carlo algorithm
that on the same input as in \Cref{thm:Main} decides whether $X$ has a $\GL$-binomial ideal, and whose output is correct with probability $\geq 1-\epsilon$.
\end{cor}
{{We give a deterministic variant of this algorithm in \Cref{alg:Binom2}.}}

\subsection*{Applications to staged tree models} 

Staged tree models are discrete models that represent dependency relations among events and are realized by a staged tree, that is, a rooted directed tree with colored nodes and labeled edges. Algebraically, the model is the image closure of the polynomial map in \eqref{eq:parametrization_for-staged_trees} given by the product of edge labels along root-to-leaf paths. 

Binary one-stage tree models  are all $\GL$-binomial by \cite[Theorem~6.5]{gorgen2022staged}. However, not all binary staged trees have $\GL$-binomial ideal by \cite[Example~25]{Maraj23} of a binary tree model with $4$ stages whose ideal is not $\GL$-binomial, and  not all ternary one-stage trees have $\GL$-binomial ideal by \cite[Example~26]{Maraj23} which shows that the ternary one-stage tree $\mathcal{C}_{3,4}$ of depth $4$ (the first tree in \Cref{table:caterpillar2}) does not have $\GL$-binomial ideal.

We prove (\Cref{ex:binary two-stage trees}) that $\GL$-binomiality breaks as soon as two stages are involved: we test this property for all  $127$ binary staged tree models with $2$ stages of depth at most $3$, and find that the $4$ models in \Cref{table:binary_two_staged} are not $\GL$-binomial.  
We also tested $\GL$-binomiality for the~$33$ extensions of the stage tree $\mathcal{C}_{3,4}$ by at most two internal nodes (\Cref{ex:caterpillar}): exactly~$16$ of them, listed in \Cref{table:caterpillar2}, are not $\GL$-binomial.

\subsection*{Applications to colored Gaussian graphical models} 
{{
These are generalizations of Gaussian graphical models, with symmetries in the entries of the concentration (precision) matrix, encoded in colorings of the graph. Its reciprocal variety is the image closure of the parametrization in \eqref{eq:colored graph} given by maximal minors of these symmetric matrices. 
We say that a model is $\GL$-binomial when its reciprocal variety has a $\GL$-binomial ideal.

Experimental results of \cite[Section~6]{Kahle25}  on  graphical Gaussian models, imply that Gaussian graphical models are $\GL$-binomial if and only if the graph is a block graph if and only if it is binomial in the original coordinates.
Other works \cite{biaggi2025binomiality,coons2023,cardwell2024toric}
provide classes of colored Gaussian graphical models on block graphs that are binomial in the original coordinates. In particular, in \cite{biaggi2025binomiality} is proven that a colored Gaussian graphical model is binomial in the original coordinates if and only if the graph is  vertex regular, edge regular, and vertex triangle regular.

We show that colors in a graph can make or break $\GL$-binomiality by testing all colorings of the (block) complete graph $K_4$  and of the (non-block) cycle $C_4$ on $4$ nodes.  More precisely, we find  there are $8$ colorings of $C_4$ (\Cref{ex:C4}), recorded in \Cref{fig:C4}, that give $\GL$-binomial models, and  that only $40$ out of the $215$ colorings of $K_4$  give $\GL$-binomial models (\Cref{ex:K4}). Moreover, among these $40$ colorings of $K_4$ there are $6$ colored graphs that are non-vertex regular, non-edge regular and not triangle regular (\Cref{fig:K4}), one  serving as a counterexample to \cite[Conjecture~22]{cardwell2024toric}, which conjectures that $\GL$-binomial BMT-derived models derived from phylogenetic trees with colored nodes must be vertex regular (\Cref{ex: conjecture}). 

 Lastly, we prove that in the extreme case when all vertices have one color and all edges have one color, the model is always $\GL$-binomial, with reciprocal variety realized by a rational normal curve (\Cref{thm: two colors}). 
}}

\subsection*{Symmetry Lie algebras of curves}

{{We treat the case of parametrized curves, that is, when~$\phi = (f_1 ,\ldots, f_n)$ is an $n$-tuple of univariate rational functions. In this case, Luroth's Theorem guarantees that the curve is \emph{rational}: in particular, up to reparameterization, the generic fiber of $\phi$ is a single point. 

We completely classify the possible symmetry Lie algebras of  $X$ in 
 \Cref{thm: Curves with nonzero symmetry}. More precisely, we show that $\frakg_X$ is either trivial, the span of a diagonalizable matrix, or  a $2$-dimensional parabolic subalgebra of $\mathfrak{sl}_2$. In the last case, $X$ is an affine rational normal curve. In particular, whenever $\frakg_X \neq \{0\}$, $X$ has a $\GL$-binomial ideal, and up to a linear change of coordinates, the components of the parametrization of $X$ are powers of a fixed rational~function.
}}

\subsection*{Symmetry Lie algebras of secant varieties}

Let $X \subseteq \CC^n$ be an irreducible variety that is closed
under scalar multiplication, that is, {{$\lambda q \in X$ for  $q \in X$ and $\lambda \in \CC$. The $r$-th \emph{secant variety} of $X$ is
$\sigma_r(X) \coloneq \overline{\{q_1+\cdots+q_r \mid q_i \in X\}} \subseteq \CC^n.$}}
Note that $\dim(\sigma_r(X)) \leq \min\{n,r \dim X \}$. The symmetry Lie group of $X$ is contained in that of $\sigma_r(X)$ and this containment may be strict: for instance, if $X$ is not a linear space and $r$ is large enough that $\sigma_r(X)$ is the linear span of $X$.

{{We show (\Cref{thm: secant varieties lie algebra}) that}} when $\sigma_{r+1}(X)$ has the expected dimension $(r+1)\dim X $, the symmetry Lie algebras of $X$ and $\sigma_r(X)$ coincide, and consequently so do the identity components of their symmetry Lie groups. {{The proof of \Cref{thm: secant varieties lie algebra} does not require $X$ to be unirational. It relies on the natural ``parametrization'' from  $X^r$ to the $r$-th secant variety~$\sigma_r(X)$ of $X$, sending an $r$-tuple of points to their sum. 
 }}

\subsection*{Organization of this paper}
 The paper is organized as follows.  In \Cref{sec:GeomChar} we prove a geometric characterization of $\frakg_X$ via the tangent spaces at points of 
$X$, which underlies \Cref{alg:Symmetry} and many other results in this paper. 
In \Cref{sec:CompSym} we prove \Cref{thm:Main},
and in \Cref{sec:Binomial} we prove \Cref{cor:Binomial}. In
\Cref{sec:DeRandom} we derandomize \Cref{alg:Symmetry}
and \Cref{alg:Binom}.
 In \Cref{sec:computational experiments} we report on computational experiments with staged tree models and colored Gaussian graphical models, and prove \Cref{thm: two colors}. In \Cref{sec:Curves} we characterize rational curves $X$ that satisfy  $\frakg_X \neq \{0\}$, showing that this occurs precisely when, up to a linear change of coordinates, the components of the parametrization are powers of a fixed rational function.  Finally, in \Cref{sec:Secants} we prove
\Cref{thm: secant varieties lie algebra} and apply it to secant varieties of Segre-Veronese embeddings.

{{We implemented \Cref{alg:Symmetry} and \Cref{alg:Binom} in SageMath.
The implementation,  together with the computational
experiments from \Cref{sec:computational experiments}, is available at
\begin{center}
    \url{https://github.com/B-Biaggi/symmetryLieAlgebra}.
\end{center}}}

\section{A geometric characterization of the symmetry Lie algebra} \label{sec:GeomChar} 

We provide a characterization of the symmetry Lie algebra of a variety $X$
as the algebra of linear maps that map each point of $X$ into its tangent
space to $X$.  Geometrically, this means that elements of the Lie algebra
define vector fields on $X$, and the corresponding complex-analytic
one-parameter subgroups will define the integral curves of such vector
fields, hence mapping $X$ to itself.

In order to stay within the purely algebraic framework, rather than
working with integral curves, we prove the result using a different
characterization of the symmetry group and hence of its Lie algebra. We
recall the following result  which is a combination of \cite[Lemma
3.1]{GHL}  and \cite[Lemma 7.4]{Bor91}; see also the discussion of
\cite[\S3.1]{GHL}.

\begin{prop}\label{prop: lie algebra via ideal}
Let $X \subseteq \CC^n$ be an algebraic variety and let $\frakg_X$ be its symmetry Lie algebra. Let
$d \geq 0$ be an integer such that $I(X)_{\leq d}$ generates the vanishing ideal $I(X)
\subseteq \CC[y_1 ,\ldots, y_n]$ of $X$. Then
\begin{align*}
\frakg_X &= \{ A \in \frakgl_n \mid A.f \in I(X)_{\leq d} \text{ for all } f \in I(X)_{\leq d} \}
\\ &= \{ A \in \frakgl_n \mid A.f \in I(X) \text{ for all } f \in I(X) \}.
\end{align*}
\end{prop}

Here $A.f$ stands for the usual action of $\frakgl_n$ on the polynomial ring 
$\CC[y_1, \ldots, y_n]$ via derivations; explicitly, $(A.f)(q)$ is the coefficient 
of $\epsilon$ in $f((Id_n - \epsilon A).q)$; see also \cite[Proposition~4.2]{Maraj23}. 
Using this characterization at the level of ideals, we obtain the following 
infinitesimal characterization in terms of tangent spaces.

{{ \begin{thm}\label{thm: frakg via tangent spaces}
Let $X \subseteq \CC^n$ be an algebraic variety, let $S \subseteq X$ be a 
Zariski-dense subset, and let $T_q X$ denote the tangent space of $X$ at $q$. 
Then $
\frakg_X = \{ A \in \frakgl_n \mid A.q \in T_q X \text{ for all } q \in S \}.$
\end{thm}
\begin{proof}
If $A \in \frakg_X$, then $A.f \in I(X)$ for all $f \in I(X)$ by 
\Cref{prop: lie algebra via ideal}. Therefore 
$
0=(A.f)(q)=f(q-\epsilon A.q) \mod (\epsilon^2).
$
This means precisely $A.q \in T_qX$.
Conversely, if $A.q \in T_qX$ for all $q \in S$, then $(A.f)(q) = 0$ for 
all $f \in I(X)$ and $q \in S$. Since $S$ is Zariski dense in $X$, $A.f$ 
vanishes on $X$, so $A.f \in I(X)$, and therefore $A \in \frakg_X$ by 
\Cref{prop: lie algebra via ideal}.
\end{proof}

We use \Cref{thm: frakg via tangent spaces} to give a more practical 
characterization of $\frakg_X$ via the Jacobian of $\phi$. For  rational map $\phi=(f_1 ,\ldots, f_n )\colon  \CC^m \dto \CC^n$, the  \emph{Jacobian map} of $\phi$ is 
\begin{align*}
J_\phi\colon \CC^m &\dto \CC^{n \times m} \\
		 p &\mapsto \left( \frac{\partial f_i}{\partial x_j}(p) \right)_{i,j}.
\end{align*} 
Define $U_\phi\coloneq \{p \in \CC^m \mid \rk J_\phi(p) = \dim X\}$.

\begin{cor}\label{cor: symmetry via jacobian}
Let $\phi = (f_1,\ldots,f_n)\colon \CC^m \dto \CC^n$ be a rational map 
and $X \coloneq \overline{\im(\phi)}$. Then~$U_\phi$ contains a Zariski open dense subset of $\CC^m$, and
\begin{align}\label{eq:Uphi}
\frakg_X = \{ A \in \frakgl_n \mid A.\phi(p) \in \im J_\phi(p) 
             \text{ for all } p \in U_\phi \}.
\end{align}
\end{cor}

\begin{proof}
First note that $U_\phi$ is Zariski open and non-empty. Its complement is the union of the set of points where $\phi$ is not defined, which is given by the vanishing of the denominators appearing in $\phi$, and the set of points where $\rk J_\phi(p) < \dim X$, given by the vanishing of minors of $J_\phi$. Moreover, if $\phi$ is defined at $p$ and $\phi(p)$ is a smooth point of $X$, then $p \in U_\phi$. This shows that $U_\phi$ is open and non-empty.

For any $p \in U_\phi$, we have $\im J_\phi(p) \subseteq T_{\phi(p)}X$ with equality if $\phi(p)$ is smooth. By \Cref{thm: frakg via tangent spaces}, we obtain the inclusion $\supseteq$ in \eqref{eq:Uphi}. On the other hand, the right-hand side of \eqref{eq:Uphi} does not change if one restricts the subset $U_\phi' \subseteq U_\phi \cap \{ p \in \CC^m \mid \phi(p) \text{ is smooth}\}$; on this subset $\im J_\phi(p) = T_{\phi(p)}X$ showing that the inclusion $\subseteq$ holds as well.
\end{proof}
}}

For each $p \in U_\phi$, the condition that $A.\phi(p)$ lies in 
the column space of $J_\phi(p)$ defines a system of linear equations 
for the entries of $A$. \Cref{cor: symmetry via jacobian} shows that, 
as $p$ varies over any dense subset $S$ of $U_\phi$, these linear 
equations cut out $\frakg_X$. Since any nonempty Zariski open subset 
of $\CC^m$ contains rational points, we may take $S = U_\phi \cap \QQ^m$.   In particular, if the coefficients of the $f_i$ all lie in $\QQ$, then 
the linear system for $\frakg_X$ is defined over $\QQ$. Moreover, since 
$\frakg_X$ is a finite-dimensional vector space, the system stabilizes 
after finitely many points: there exist $p_1,\ldots,p_M \in U_\phi \cap \QQ^m$ 
such that the equations at $p_1,\ldots,p_M$ alone already cut out~$\frakg_X$. This observation is used to design \Cref{alg:Symmetry} in the next section.

The characterization of $\frakg_X$ in \Cref{thm: frakg via tangent spaces} applies to any variety $X$. Special geometric features of $X$ can however simplify the computation considerably. We record two such cases: when~$X$ 
is \emph{linearly degenerate}, that is, contained in a proper linear 
subspace of $\CC^n$, and when~$X$ is a \emph{cylinder}, that is, 
$X = \pi^{-1}\pi(X)$ where $\pi\colon \CC^n \to \CC^n/L$ is the 
quotient by some linear subspace $L$.
{{\begin{prop}\label{prop: linearly degenerate}
Let $X \subseteq \CC^n$ be an algebraic variety contained in a
linear subspace $U \subseteq \CC^n$. Let $K_U \subseteq \frakgl_n$ be the
subspace of linear maps $A\colon \CC^n \to \CC^n$ with $A|_U = 0$.~Then~$K_U  \subseteq \frakg_X$. Moreover, if $U$ is spanned by $X$, then 
$\frakg_X$ is the preimage in $\frakgl_n$ of the symmetry Lie algebra~$\frakg_{X,U}$ of $X$ in $\frakgl(U)$, under the natural map $\frakgl_n \to \frakgl_n/K_U \cong \Hom(U,\CC^n)$.
\end{prop}
\begin{proof}
For the inclusion $K_U \subseteq \frakg_X$, note that any $A \in K_U$ 
satisfies $A.q = 0 \in T_qX$ for all~$q \in X \subseteq U$, so 
\Cref{thm: frakg via tangent spaces} gives $A \in \frakg_X$.

For the second statement, note that $\frakgl_n/K_U \cong \Hom(U,\CC^n)$, 
and $\frakgl(U) \subseteq \Hom(U,\CC^n)$ consists of those maps 
sending $U$ to itself. It suffices to show that for any $A \in \frakg_X$, 
the restriction $A|_U$ maps $U$ into $U$. By \Cref{thm: frakg via tangent spaces}, for every $q \in X$ we have $A.q \in T_qX \subseteq U$, since $X \subseteq U$ implies $T_qX \subseteq U$.
 Thus $A.q \in U$ for all $q \in X$, and 
since $A$ is linear and~$U = \langle X\rangle$, we conclude $A(U) \subseteq U$.
\end{proof}}} 

{{\begin{prop}\label{prop: cylinders}
Let $L \subseteq \CC^n$ be a linear subspace and let $\pi\colon \CC^n 
\to \CC^n/L$ be the quotient map. Let $X \subseteq \CC^n$ be an algebraic variety which is a cylinder
over $L$, that is $X = \pi^{-1}(\pi(X))$. Let $P_L
\subseteq \frakgl_n$ be the subspace of linear maps $A\colon \CC^n \to
\CC^n$ with $\im(A) \subseteq L$. Then $P_L \subseteq \frakg_X$.
Moreover, if $\pi(X)$ is not a cylinder, 
then $\frakg_X$ is the preimage in $\frakgl_n$ of the symmetry Lie 
algebra~$\frakg_{\pi(X)}$ of $\pi(X)$ in $\frakgl(\CC^n/L)$, under 
the natural map $\frakgl_n \to \frakgl_n/P_L \cong \Hom(\CC^n, \CC^n/L)$.
\end{prop}
\begin{proof}
For the inclusion $P_L \subseteq \frakg_X$, note that any $A \in P_L$ 
satisfies $A.q \in L \subseteq T_qX$ for all $q \in X$ since $X$ is 
a cylinder over $L$, so \Cref{thm: frakg via tangent spaces} gives 
$A \in \frakg_X$.

For the second statement, note that $\frakgl_n/P_L \simeq 
\Hom(\CC^n, \CC^n/L)$, and it suffices to show that for any 
$A \in \frakg_X$, viewed as an element of $\Hom(\CC^n,\CC^n/L)$, 
the restriction $A|_L$ is zero. For every $q \in X$ and $v \in L$ 
we have $q+v \in X$ and $T_{q+v}X = T_qX$, so $T_{\pi(q)}\pi(X) = 
T_qX/L$. By \Cref{thm: frakg via tangent spaces}, both $A.q$ and 
$A.(q+v)$ lie in $T_{\pi(q)}\pi(X)$, so by linearity of $A$, their 
difference~$A.v$ also lies in $T_{\pi(q)}\pi(X)$ for every $q \in X$. 
Thus, $A.v$ lies in the intersection of all tangent spaces of $\pi(X)$, 
which is trivial since $\pi(X)$ is not a cylinder. Hence $A|_L = 0$, as desired.
\end{proof}
}}

\Cref{prop: linearly degenerate} and \Cref{prop: cylinders} are \emph{dual} to each other  in a precise sense, which we record here. 
\begin{re}
It is a classical fact that the dual variety of $X$ is a cylinder 
if and only if $X$ is contained in a hyperplane; this is usually 
stated in the projective setting, where cylinders are replaced by 
cones. Moreover, the symmetry Lie algebra is preserved under duality, 
that is, the symmetry Lie algebras of $X$ and its dual variety $X^\vee$ 
are identified by the natural isomorphism $\frakgl(\CC^n) \simeq \frakgl((\CC^n)^*)$; see \cite[Proposition~3.11]{GHL}.
 In summary, the statement of \Cref{prop: cylinders} for a cylinder $X$ can be recovered from \Cref{prop: linearly degenerate} applied to $X^\vee$. 
\end{re}
{{We conclude this section with a remark about affine symmetries, that is the group of transformations on $\CC^n$ generated by $\GL_n$ and translations; see also \cite[Section~5]{Kahle25}.}} 
\begin{re}\label{rmk: affine symmetries}
Let $X \subseteq \CC^n$ be an algebraic variety. If $X$ is not contained in an affine hyperplane, then there is a bijection between affine symmetries of $X$, that is, linear transformations in $\GL_n \ltimes \CC^n$ which stabilize $X$, and linear symmetries of $X \times \{1\}$ in $\GL(\CC^n \oplus \CC^1)$. The bijection is given by 

\begin{align*}
\GL_n \ltimes \CC^n &\to  \GL(\CC^n \oplus \CC^1) \\
({{q}} \mapsto g{{q}} + v) &\mapsto 
\begin{pmatrix}
g & v \\
0 & 1
\end{pmatrix}
\end{align*}
Any symmetry of $X$ in $\GL_n \ltimes \CC^n$ maps to a symmetry of 
$X \times \{1\}$ in $\GL(\CC^n \oplus \CC^1)$ under this map. 
Conversely, assume $\left(\begin{smallmatrix} g & v \\ w^T & c 
\end{smallmatrix}\right)$ stabilizes $X \times \{1\}$. Then for 
every ${{q}} \in X$ we have~$w^T {{q}} + c = 1$. Since $X$ is not contained 
in an affine hyperplane, we conclude $w = 0$ and $c = 1$.
\end{re}

\section{A polynomial-time algorithm for computing the symmetry Lie algebra}
\label{sec:CompSym}
This section presents \Cref{alg:Symmetry}, a polynomial-time probabilistic algorithm that computes 
$\frakg_X$ of a variety $X$ given a rational parametrization 
 $\phi=(f_1,\ldots,f_n)$~of~$X$. 

\begin{algorithm}
\caption{{Computes a basis for $\frakg_X$ where 
$X = \overline{\im(\phi)}$.}}\label{alg:Symmetry}
\begin{algorithmic}[1]
\Require The parametrization $\phi(x) = (f_1 ,\ldots, f_n) \in \QQ(x_1 ,\ldots, x_m)^n$ {{and $\epsilon > 0$}};
\Ensure A basis of the symmetry Lie algebra $\frakg_X$;   \Comment{{{with probability $\geq 1-\epsilon$}}}
\State $d \gets$ an upper bound on the degrees of all  {{$a_i, b_i \in \QQ[x_1,\ldots,x_m]$, where
 $f_i = a_i/b_i$}};
\State $N \gets \lceil \epsilon^{-1} \cdot 2 \cdot n^3 \cdot d \cdot
(2m+1) \rceil$;
\State $L \gets \emptyset$;
\State $s \gets 0$;
\While{$s<n^2$}
	\State $p \gets$ uniformly random point in $\{1 ,\ldots, N\}^m$; 
	\State $q \gets \phi(p)$ (if undefined, \Return ``fail'');
	\State $J \gets \left(\frac{\partial f_i}{\partial
	x_j}(p)\right)_{ij} \in \QQ^{n \times m}$;
	\State $E \gets \text{a basis for the orthogonal complement in
	$\QQ^n$ of the column space of } J$;

	\State $L \gets L \cup \{e^T A.q \mid e \in
	E\}$;
     \Comment{{{linear forms in the entries of $A \in \frakgl_n$}}}
	\State $s \gets s+1$;
\EndWhile
\State \Return A basis of the solution space of the linear system  $L$;
\end{algorithmic}
\end{algorithm}

\begin{prop} \label{prop:Symmetry}
\Cref{alg:Symmetry} returns a basis for $\frakg_X$ with probability at least $1-\epsilon$.
\end{prop}

\begin{proof}
Since $\frakg_X \subseteq \frakgl_n$ has dimension at most $n^2$, 
a choice of $n^2$ 
sufficiently general points~$p_1,\ldots,p_{n^2}$ in $U_\phi$ is sufficient for the computation of $\frakg_X$ in  \Cref{cor: symmetry via jacobian}, that is,  
\[
\frakg_X = \{ A \in \frakgl_n \mid A.\phi(p_s) \in \im J_\phi(p_s) 
\ \text{ for } s=1,\ldots,n^2 \}.
\]
A necessary and sufficient genericity condition is that the rank of the linear system 
\begin{equation}\label{eqn: linear system for frakg}
A.\phi(p_s) \in \im {{J_\phi}}(p_s) \text{ for } s=1 ,\ldots, n^2
\end{equation}
equals $n^2 - \ell$, where $\ell = \dim \frakg_X$. This condition defines an open subset of $U_\phi^{\times n^2}$ hence of 
$(\CC^m)^{\times n^2}$, whose complement is contained in a 
hypersurface in $(\CC^m)^{\times n^2}$ of degree at most 
$2n^3d(2m+1)$. Assuming this degree bound, the Schwartz-Zippel Lemma implies that 
the probability that a uniformly random point in 
$(p_1,\ldots,p_{n^2})\in \{1,\ldots,N\}^{mn^2}$ lies in this hypersurface is at most 
$2n^3d(2m+1)/N$. Taking $N \geq \epsilon^{-1} \cdot 2n^3d(2m+1)$ 
ensures that with probability at least $1-\epsilon$, the sampled 
points $(p_1,\ldots,p_{n^2})$ avoid this hypersurface, and
\Cref{alg:Symmetry} successfully computes $\frakg_X$.

The rest of this proof establishes the degree bound $2n^3d(2m+1)$. 

\begin{enumerate}[(i)]
\item The condition that $\phi$ is defined at $p$ fails on the 
hypersurface defined by the vanishing of $b \coloneq b_1\cdots b_n$, 
which has degree at most $nd$. The condition failing at one or more 
of $p_1,\ldots,p_{n^2}$ is therefore contained in a hypersurface of 
degree at most $n^3d$ in $(\CC^m)^{\times n^2}$.

\item The condition $\rk J_\phi(p) = \dim X$ fails on the 
hypersurface defined by the vanishing of the minors of $J_\phi(p)$ 
of size $\dim X$, and it suffices to consider a single fixed such 
minor. The entries of $J_\phi(p)$ are partial derivatives 
$\frac{\partial f_i}{\partial x_j}$, each of the form 
$\frac{1}{b^2}$ times a polynomial of degree at most $2dn$. 
Ignoring the common denominator, the vanishing of a fixed minor 
of size $\dim X$ is equivalent to the vanishing of a polynomial 
of degree at most~$2dn \cdot \dim X \leq 2dnm$. The condition 
failing at one or more of $p_1,\ldots,p_{n^2}$ is contained in 
a hypersurface of degree at most $2n^3dm$ in $(\CC^m)^{\times n^2}$.

\item The complement of the condition that the rank of the linear system in \eqref{eqn: linear system for frakg} is equal to $n^2 - \ell$ is equivalent to the vanishing of all minors of size~$n^2 - \ell$ of the matrix of the linear system. As in the previous case, this is contained in the hypersurface defined by the vanishing of a single fixed minor. The equations of the linear system are the minors of size $(1+ \dim X)$ of the matrix $[{{J_\phi}} (p) | A.\phi(p)]$ which involve the last column. The entries of $A.\phi(p)$ are of the form $\frac{1}{b}$ multiplied by a polynomial of degree at most $nd$ and by a linear form in the entries of $A$. As in the previous case, we may ignore denominators, and deduce that the coefficients of the linear system in \eqref{eqn: linear system for frakg} are polynomials of degree at most $nd + 2nd {{\cdot (\dim X)}}$, which is bounded above by $nd(2m+1)$. Therefore, a fixed minor of size~$n^2 - \ell$ of the matrix of the linear system in \eqref{eqn: linear system for frakg} is a polynomial of degree at most $(n^2 - \ell) nd(2m+1)$, which {{in turn}} is bounded above by $n^3d(2m+1)$. 
\end{enumerate} 
{{Adding the three contributions together, we obtain that the complement of the required sufficiently general condition  is contained in a hypersurface of degree at most~$ 2n^3 d (2m+1)$.}} 
\end{proof}

To ensure the input is nontrivial and that all variables appear, 
we make the following technical assumptions on the encoding of the 
input of \Cref{alg:Symmetry}.
\begin{enumerate}
\item If the $f_i$ are given in a standard encoding as lists of 
pairs $(c,\alpha)$ of coefficients and exponent vectors, then at 
least one $f_i$ is nonzero, that is, represented by a nonempty list.
\item If the $f_i$ are given as arithmetic circuits, then every 
variable $x_j$ appears as the label of at least one node in at least one of the $n$ circuits.
\end{enumerate}
Under these technical assumptions, \Cref{alg:Symmetry} runs in 
polynomial time. This is the content of the next result.
\begin{prop}\label{prop:running time for Algorithm 1}
The integer $d$ in Step 1 of \Cref{alg:Symmetry} can be chosen so that \Cref{alg:Symmetry} has polynomial running time in the bit length of the input and in $\log(1/\epsilon)$. 
\end{prop}
\begin{proof}
Under the non-triviality assumption{{s}} (1) and (2), both $m$ and $n$ 
are bounded by a linear polynomial in the bit length of the input. 
For $m$:  in the standard encoding, some $f_i$
has a term represented by a pair $(c,\alpha)$, and since  $\alpha$ has $m$ components; $m$ is bounded by the size of $\alpha$;  in the circuit encoding,
 every variable appears at least once, so $m$ is bounded by the number of nodes. For $n$:
the input consists of $n$ rational functions, so $n$ is bounded 
by the size of the input list.

The bound $d$ on the degrees of numerators and denominators of the $f_i$ can be chosen polynomially in the bit length of the input. In the standard encoding, one may take $d = \max \{ \deg(a_i) , \deg(b_i)  \mid i = 1 ,\ldots, n \}$, which in fact is polynomial in the input length. In the circuit encoding, one may choose $d = 2^s$, where $s$ is the size of the largest circuit among the ones computing the $f_i$, and so, the bit length of $d$ is linear in $s$. 
As a consequence, the constant $N$ in \Cref{alg:Symmetry} has bit length polynomial in the bit length of the input. 

The loop in \Cref{alg:Symmetry} runs $n^2$ times. In each iteration, 
the computation of $\phi(p)$ and of~$J_\phi(p)$ can be performed in 
polynomial time: in the standard encoding this is immediate, as 
derivatives can be computed symbolically; in the circuit encoding, 
this follows from \cite[Theorem~2]{BaSt83}. Computing a basis for 
the orthogonal complement of the column space of $J_\phi(p)$ and 
updating $L$ also take polynomial time. The resulting linear system 
has at most $(m+1)n^2$ equations in $n^2$ unknowns and can therefore 
be solved in polynomial time. By \Cref{cor: symmetry via jacobian}, 
a basis of its solution space is a basis of $\frakg_X$. 

Finally, the only dependence on $\epsilon$ is through the choice of $N$. 
Since $\log N$ grows linearly in~$\log(1/\epsilon)$, this contributes 
a $\log(1/\epsilon)$ term to the running time, which remains 
polynomial in the input size and in $\log(1/\epsilon)$.
\end{proof}

This concludes  the proof of \Cref{thm:Main}. 

\begin{re} \label{re:Adapt}
In a practical implementation of \Cref{alg:Symmetry} one may
first try to guess $k\coloneq \dim X $ by computing $\rk {{J_{\phi}}}(p)$ for a number of
values of $p$  at which $\phi$ is defined, and taking the 
maximum rank observed. In each iteration of the loop, one may 
resample $p$ until $\phi$ is defined at $p$ and $\rk J_\phi(p) = k$.  Furthermore,
instead of iterating the loop $n^2$ times, one may iterate it until $L$
stabilizes, or stabilizes for a number of steps. In this manner, the
output of the algorithm is never ``fail'', and moreover, if the initial
guess for $k$ is correct, then the output is always a subspace containing~$\frakg_\QQ$. The running 
time becomes a random variable, but its expected value remains 
polynomial.
\end{re}

\begin{re}
\Cref{alg:Symmetry} shows that the problem of deciding
whether $\dim(\frakg_{\QQ}) \geq \ell$ for some input $\ell$ is in the {{complexity class BPP (bounded-error probabilistic polynomial time).}} We do not know 
whether this problem is in {{RP (randomized polynomial time)}} or in 
co-RP.  The point here is that \Cref{alg:Symmetry}
can output both an over-estimate of $\frakg_\QQ$ (e.g.~if the points
sampled satisfy $\rk {{J_{\phi}}}(p)=\dim X $ but are collectively in too special
position) or an under-estimate of~$\frakg_\QQ$ (e.g.~if all sampled
points $p$ satisfy $\rk {{J_{\phi}}}(p)=0<\dim X $).
\end{re}

\begin{re}\label{re: linearly degenerate}
If $X$ is either linearly degenerate or a cylinder, \Cref{prop: linearly degenerate} and \Cref{prop: cylinders} reduce the computation of $\frakg_X$ to the computation of the symmetry Lie algebra of a variety in a smaller ambient space. In the case of linearly degenerate varieties, the subspace $\langle X \rangle$ can be computed efficiently by sampling points on $X$ and computing their linear span. In the case of cylinders, the linear space $L$ over which one should quotient $X$ can be computed by computing the linear span of the dual variety of $X$; points of the dual variety can be sampled efficiently by sampling a point ${{q\in }}X$ and then a random hyperplane containing $T_{{q}} X$. These procedures return $L$ and $\langle X \rangle$ with high probability. In practice, an efficient implementation of \Cref{alg:Symmetry} should first check whether $X$ is linearly degenerate or a cylinder, and reduce to the non-degenerate case accordingly.
\end{re}
\begin{re}
The output of \Cref{alg:Symmetry} can be thought of as the Lie
algebra of those vector fields $\sum_{i=1}^n h_i \frac{\partial}{\partial y_i}$
on $\CC^n$ that are tangent to $X$ and for which the $h_i$ are homogeneous linear polynomials.  One can easily adapt \Cref{alg:Symmetry} to the setting of affine linear symmetries, as in \Cref{rmk: affine symmetries}: in this case the elements of the resulting Lie algebra can be interpreted as vector fields tangent to $X$ with affine linear coefficients $h_i$. Moreover, fixing any degree $e$, a similar algorithm yields a basis of the space of vector fields tangent to $X$ where the $h_i$ have degrees $\leq e$. But for $e \geq 2$ these typically do not form a Lie algebra, since the commutator of two such vector fields may have coefficient polynomials of degree up to $2e-1>e$. Also, these vector fields may not come from the action of any finite-dimensional algebraic group.
\end{re}

\section{Testing $\GL$-binomiality} \label{sec:Binomial}
We say that an ideal $I \subseteq \CC[y_1 \vvirg y_n]$ is \emph{$\GL$-binomial} if there exists $g \in \GL_n$ such that $g I$ is generated by binomials, that is, polynomials of the form $y^\alpha - cy^\beta$ where~$\alpha,\beta \in \ZZ_{\geq 0}^n$ and~$c \in \CC$. A variety $X \subseteq \CC^n$ is \emph{$\GL$-binomial} if $I(X)$ is a $\GL$-binomial ideal. Note that we 
allow $c = 0$, consistent with \cite{ES96}.

This section presents {{\Cref{alg:Binom}}},  a polynomial-time probabilistic algorithm that decides, given a parametrization of a variety $X \subseteq \CC^n$, whether $X$ is \emph{$\GL$-binomial}.  The algorithm combines 
\Cref{alg:Symmetry} with ideas from \cite{Kahle25} and known facts 
about algebraic groups.

\begin{algorithm}
\caption{Tests probabilistically whether $X=\overline{\im\phi}$ is $\GL$-binomial.}\label{alg:Binom}
\begin{algorithmic}[1]
\Require $\phi=(f_1 ,\ldots, f_n) \in \QQ(x_1 ,\ldots, x_m)^n$ {{and $\epsilon>0$}};
\Ensure {{whether $I$ binomial after a linear coordinate change}};
 \Comment{{{with probability $\geq
(1-\epsilon)^3$}}} 
\State $M \gets \lceil \epsilon^{-1} \cdot n^2 \rceil$; 
\State ${{\mathcal{B}}} \gets $ the output of~\Cref{alg:Symmetry}; \Comment{{{with probability $\geq
1-\epsilon$}}} 
\If{$\mathcal{B} = $ ``fail''} \Return ``fail''; \EndIf
\State $L \gets $ linear space spanned by ${{\mathcal{B}}}$;
\State $A \gets $ a linear combination of ${{\mathcal{B}}}$ with 
uniformly random coefficients in $\{1 ,\ldots, M\}$;
\State $A_s \gets $ semisimple part of $A$; 
\State ${{\mathcal{C}}} \gets $ a basis of $\{D \in L \mid [A_s,D]=0\}$;
\For{$D \in {{\mathcal{C}}}$}
	\State $D \gets $ semisimple part of $D$;
\EndFor;
\State $N \gets \lceil \epsilon^{-1} \cdot n \cdot d \cdot (1+2m+n) \rceil$;
\State $p \gets $ uniformly random point in $\{1 ,\ldots, N\}^m$;
\State $q \gets \phi(p)$; 
\State{if undefined, \Return ``fail''}
\State $J \gets (\frac{\partial f_i}{\partial x_j})(p)$;
\State \If{$\rk(J) = \dim \langle \{Dq \mid D \in {{\mathcal{C}}}\}\rangle$}
\State $b \gets $ ``true''; 
\Else \State $b \gets $ ``false''; \EndIf;
\State \Return $b$ 
\end{algorithmic}
\end{algorithm}

The correctness of \Cref{alg:Binom} uses \Cref{lm:Cartan}. {{For a matrix $A \in \frakgl_n$, the adjoint 
map $\ad(A)\colon \frakgl_n \to \frakgl_n$ is defined by 
$\ad(A)(B) = [A,B] = AB - BA$. The Jordan decomposition guarantees 
that any $A$ can be written uniquely as $A = A_s + A_n$, where $A_s$ 
is semisimple, $A_n$ is nilpotent, and $A_s$ has the same eigenvalues as $A$.}}

\begin{lm} \label{lm:Cartan}
Let $H$ be a closed connected subgroup of $\GL_n$, $\mathfrak{h}$ its Lie algebra, and $A \in \mathfrak{h}$ such that the generalized eigenspace of $\ad(A)\colon \mathfrak{h} \to \mathfrak{h}$ with
eigenvalue $0$ has the minimal dimension among all elements of
$\mathfrak{h}$. Let $A=A_s+A_n$ be the Jordan decomposition of~$A$. Then the centralizer
$Z \coloneq \{D \in \mathfrak{h} \mid [A_s,D]=0\} $
is the Lie algebra of a Cartan subgroup of $H$, and the set 
$ Z_s \coloneq \{D_s \mid D{{=D_s+D_n}} \in Z\} $
is the Lie algebra of a maximal torus of $H$.
\end{lm}

\begin{proof}
Since $A_s$ is semisimple, $H$ has a maximal torus $T$ whose Lie algebra
contains $A_s$. Since~$\ad(A)_s=\ad(A_s)$, the eigenspace of $\ad(A_s)$
with eigenvalue $0$ is the generalized eigenspace of $\ad(A)$ with
eigenvalue $0$. By minimality of the latter, $\ad(A_s)$ is nonzero on
{{every}} $T$-weight space in $\mathfrak{h}$ corresponding to a nonzero weight. 
{{By \cite[\S12.2, Lemma]{Bor91},
the (analytic) one-parameter subgroup $t \mapsto \exp(tA)$ contains a regular
semisimple element of $T$. Then $Z$ is the Lie algebra of the centralizer of $T$, which is a Cartan subgroup of $H$ by \cite[\S12.2 Proposition]{Bor91},}} and
$Z_s$ is the Lie algebra of $T$ by \cite[\S12.1]{Bor91}.
\end{proof}

\begin{prop}\label{prop: alg2 probability}
With probability at least $(1-\epsilon)^3$, the output of \Cref{alg:Binom} is correct.
\end{prop}

\begin{proof}
The computations up to and including the for-loop are intended to produce a basis for the Lie algebra of a maximal torus $S$ of the symmetry group $G_X$. The remaining lines test whether the $S$-orbit of a sufficiently general point in $X$ is dense in $X$. Following the argument of \cite[Section 3]{Kahle25}, this condition is equivalent to $\GL$-binomiality. We now explain why this gives the correct answer {{with probability at least $(1-\epsilon)^3$}}.

We write $(\CC^*)^n$ for the subgroup of $\GL_n$ consisting of
diagonal matrices. 
If $Y \subseteq \CC^n$ is
an irreducible variety, then the vanishing ideal {{$I(Y)$}} of $Y$ is binomial
if and only if the group of diagonal symmetries
\[ T=\{t \in (\CC^*)^n \mid t Y = Y \} \]
has a dense orbit on $Y$, and this happens if and only if the identity
component $T^0$ of $T$ has a dense orbit on $Y$. All of this is well
known: if $q \in Y$ is such that $Tq$ is dense in $Y$, then the vanishing ideal of $Y$ is spanned by all binomials of the form $y^\alpha-cy^\beta$
that vanish on $q$ and such that the kernel of the character $(\CC^*)^n
\to \CC^*,\ (t_1 ,\ldots, t_n) \mapsto \prod_{i=1}^n t_i^{\alpha_i-\beta_i}$
contains $T$. Conversely, if  $I(Y)$ is binomial, 
${{\mathrm{s}(I)}} \coloneq \{i \in \{1 ,\ldots, n\} \mid y_i \not \in I(Y)\}$, and $q \in Y$
has $y_i(q) \neq 0$ for all $i \in \mathrm{s}(I)$, then $Tq$ is dense in $Y$ and
consists of all points whose coordinates $y_i, i \in \mathrm{s}(I)$
{{are nonzero. In fact, in this case $T=T^0$.}} 

Assume that the vanishing ideal of $Y\coloneq gX$ is binomial for some $g \in
\GL_n$. Define $T=T^0$ as above.  Then $R\coloneq g^{-1} T g$
is a torus in  $G^0$ with a 
dense orbit on $X$. If we enlarge $R$ to a {\em maximal} torus in $G^0$, it still has a dense orbit on $X$. Since maximal tori in $G^0$
are conjugate by elements in $G^0$, it follows that {\em any} maximal
torus in $G^0$ has a dense orbit on $X$.

Conversely, let $S$ be any maximal torus in $G^0$. By conjugacy of maximal tori in $\GL_n$, there exists $h \in \GL_n$ such that $hSh^{-1} \subseteq (\CC^*)^n$.  Set $Z\coloneq hX$, and assume that $y_1,\ldots, y_{b}$ are identically zero on $Z$ and the remaining coordinates are not. Let $k=\dim X$. The points in $X$ with $S$-orbit of maximal dimension are exactly those $q \in X$ for which the last $n-b$ coordinates of $hq$ are nonzero. If these form a single dense $S$-orbit, then it follows that the vanishing ideal $I(Z)$ is binomial. Otherwise, by the previous  paragraph, $I(gX)$ is not binomial for any~$g\in\GL_n$. It therefore suffices to check whether a single $q \in X$ with the property that the last $n-b$ coordinates of $hq$ are nonzero has a dense $S$-orbit. Writing $\liea{s}$ for the Lie algebra of $S$, this is equivalent to $\liea{s}q$ having dimension  $k$.

Assume first that the correct answer is ``true''.  Then \Cref{alg:Binom} gives the correct output if
\begin{enumerate}
\item after the for-loop,  $\mathcal{C}$ spans the Lie algebra of a maximal torus
of $G^0$; 
\item $\phi$ is defined at $p$;
\item the rank of ${{J_{\phi}}}(p)$ is equal to $k$; and
\item $\phi(p)$ lies outside a union of $n-b$ hyperplanes.
\end{enumerate}
As in the proof of \Cref{prop:Symmetry}, item (2) holds for $p$ outside a
hypersurface of degree at most~$nd$, where $d$ is the bound on the degrees
of numerators and denominators from \Cref{alg:Symmetry}. Similarly,
items (3) and (4) hold for $p$ outside hypersurfaces of degrees (at most)~$2ndm$ and $n^2 d$ respectively. {{For items (2), (3), and (4) to hold simultaneously}}, $p$
should avoid a hypersurface of degree $nd+2ndm+n^2d=nd(1+2m+n)$. By the
Schwartz-Zippel Lemma, the choice of $N$ ensures that this happens with
probability at least $1-\epsilon$.

The call to \Cref{alg:Symmetry} gives the correct answer with 
probability at least $1-\epsilon$. Assuming $L$ is correct, it 
remains to show that item (1) holds with probability at least 
$1-\epsilon$.  By 
\Cref{lm:Cartan}, it suffices to ensure that $\ad(A )\colon \frakg \to
\frakg$ has a generalized eigenspace with eigenvalue $0$ of the minimal
possible dimension $s$. This requires the 
coefficient of $\lambda^s$ in the characteristic polynomial 
$\det(\lambda\,\id_\frakg - \ad(A))$ to be nonzero. This coefficient is a polynomial in the
entries of $A$ of degree $l-s \leq n^2$, where $l\coloneq\dim(\frakg)$. Hence
the choice of $M$ and the Schwartz-Zippel Lemma imply that (1) holds
with probability at least $1-\epsilon$, as desired. 

Assume next that the correct answer is ``false''. 
With probability at least $(1-\epsilon)^2$, the output of \Cref{alg:Symmetry} is 
correct and the for-loop produces a basis for the Lie algebra of a 
maximal torus $S$ in $\frakg_X$. Moreover, with probability $\geq 1-\epsilon$ the Jacobian $J_{\phi}(p)$ has rank $k$, and then the algorithm will return ``false'',  {{since the vectors $Dq$ for $D\in\mathcal{C}$ lie in the tangent space to the $S$-orbit of $q$, which has dimension less than $k$ since the correct answer is ``false'', so they cannot span a space of dimension $k$. The three probabilistic events are independent, giving a combined success probability of at least $(1-\epsilon)^3$.}}
\end{proof}

The following result completes the proof of \Cref{cor:Binomial}.
\begin{prop}\label{prop: polytime Algo2}
\Cref{alg:Binom} runs in polynomial time.
\end{prop}
\begin{proof}
The analysis is similar to that
of \Cref{alg:Symmetry}. The additional 
steps in \Cref{alg:Binom} are: computing a random linear combination 
of $\mathcal{B}$; computing the semisimple part of a matrix with 
rational entries, which can be done in polynomial time by 
\cite{Levelt93}; computing a basis for the centralizer 
$\{D\in L\mid [A_s,D]=0\}$, which reduces to solving a linear 
system; and evaluating the rank condition on 
$\langle\{Dq\mid D\in\mathcal{C}\}\rangle$. All of these steps 
take polynomial time.
To complete the proof of \Cref{cor:Binomial}, it suffices to replace  $\epsilon$ by $\epsilon/3$ in the discussion of \Cref{prop: alg2 probability}.
\end{proof}

Now we share some remarks on \Cref{alg:Binom}.

\begin{re}\label{rem: find g making binomial}
If the output of \Cref{alg:Binom} is ``true'', one may ask how to compute $g\in\GL_n$ such that the ideal of $gX$ is binomial. In principle, it suffices to find $g$ such that $g\mathcal{C}g^{-1}$ consists of diagonal matrices, that is, to simultaneously diagonalize the elements of $\mathcal{C}$. Such~a~$g$ may not exist over $\QQ$, but if it does, it can be found in polynomial time via factorization of polynomials over $\QQ$ \cite{LLL82}, since the eigenvalues of the elements of $\mathcal{C}$ are roots of polynomials with rational coefficients. If $g$ does not exist over $\QQ$, one can compute its entries in a suitable splitting field. However, since the splitting field of a degree-$n$ polynomial may have degree $n!$ over $\QQ$, this is potentially exponential in $n$ and can no longer be done in 
polynomial time.
\end{re}

\begin{re}
There are practical settings in which the $g$ of \Cref{rem: find g making binomial} requires a field extension of $\QQ$. In \cite{SS05}, Sturmfels and Sullivant show that the ideals of group-based models for abelian groups are $\GL$-binomial and the required change of coordinate is the Fourier transform of the group. This coordinate change requires roots of unity that are typically not in $\QQ$. Incidentally, this is a case where the $\GL$-binomiality has been exploited
in a very successful manner to find equations for~$X$; see, e.g.,~\cite{MV19,Noren12}. 
\end{re}

\begin{re}
{{A simpler problem is to decide whether the vanishing ideal of $X$ 
itself is binomial, without need of a coordinate change. }}
Indeed, in this case, we may restrict our attention
to diagonal matrices {{$D=\diag(\lambda_1 ,\ldots, \lambda_n)$, which act on $\CC^n$ by 
coordinate-wise multiplication. One computes}}
a basis $C$ for the space of such $D$ mapping sufficiently many sampled
points $q$ into $T_q X$, then samples a new point $q'$ and tests 
whether $\dim\langle\{dq'\mid d\in\mathcal{C}\}\rangle = \dim X$.
An important 
difference is that we do linear algebra in $\QQ^n$ rather than $\QQ^{n^2}$!
\end{re}

\section{Derandomising} \label{sec:DeRandom}

\Cref{alg:Symmetry} and \Cref{alg:Binom} can be 
derandomized at 
the cost of losing polynomial running time.
Throughout this section we
assume that the input of $\phi$ is given in the standard encoding, and we
carry out symbolic computations with rational functions.

\begin{algorithm}[h]
\caption{Computes deterministically $\frakg_X$, where 
$X = \overline{\im(\phi)}$.} \label{alg:Symmetry2}
\begin{algorithmic}[1]
\Require $\phi=(f_1 ,\ldots, f_n) \in \QQ(x_1 ,\ldots, x_m)^n$;
\Ensure ${{\mathcal B}= \text{ basis for }}\frakg_\QQ$;
\State $p \gets (x_1 ,\ldots, x_n)$; \Comment{The $x_i$ are variables.}
\State $q \gets \phi(p)$;
\State $J \gets \left(\frac{\partial f_i}{\partial
x_j}(p)\right)_{ij} \in \QQ(x_1 ,\ldots, x_m)^{n \times m}$ ;
\State $k \coloneq \rk_{\QQ(x_1 ,\ldots, x_m)} J$;
\State $A \gets (a_{ij})_{i,j}$; \Comment{$A$ is an $n \times
n$-matrix of variables.}
\State $E \gets \{$all $(k+1)\times(k+1)$-subdeterminants\\ \quad
\quad \quad of the $n
\times (m+1)$-matrix $(J|A.q)$ that involve the last column$\}$;
\State $F \gets \{$all coefficients of all monomials in $x$\\ \quad
\quad \quad  in the 
numerators of the elements of $E\}$; \Comment{linear forms in the entries of
$A$}
\State ${{\mathcal B}} \gets $ a basis for the common zero set of the linear forms in $F$
\State \Return ${{\mathcal B}}$;
\end{algorithmic}
\end{algorithm}

\begin{prop}
\Cref{alg:Symmetry2} computes a basis for the symmetry algebra $\frakg_X$ of 
$X =\overline{\im(\phi)}$. 
\end{prop}

\begin{proof}
For $A \in \frakgl_n$ to be in $\frakg_X$, by \Cref{thm: frakg via tangent spaces}, it is necessary and sufficient that $A.q \in T_q X$
for $q$ in some open dense subset of $X$. In the algorithm, this is tested
for nonsingular $q$ of the form $\phi(p)$ with the additional requirements
that $\rk {{J_{\phi}}}(p)=\dim X $, so that $T_q X=\im({{J_{\phi}}}(p))$. Indeed, note that
the elements of $E$ are linear forms in the $a_{ij}$ with coefficients in
$\QQ(x_1 ,\ldots, x_n)$, and so we can write each $e \in E$ as $\frac{1}{h}
\widetilde{e}$, where $\widetilde{e}$ is a linear form in the $a_{ij}$
with coefficients in $\QQ[x_1 ,\ldots, x_n]$ and where $h$ is a polynomial in $\QQ[x_1 ,\ldots, x_n]$. Regrouping terms, we regard $\widetilde{e}$ as a polynomial in the $x_i$ with coefficients that are linear forms in the $a_{ij}$. For $A \in \frakgl_n$ to satisfy $A.q \in T_q X$ for all $q$ as above, it is necessary and sufficient that the elements $\widetilde{e}$ are all identically zero, which is equivalent to saying that the entries of $A$ satisfy the linear forms~in~$F$.
\end{proof}

\begin{algorithm}[h]
\caption{Tests deterministically whether $X=\overline{\im(\phi)}$ has $\GL$-binomial ideal.}\label{alg:Binom2}
\begin{algorithmic}[1]
\Require $\phi=(f_1 ,\ldots, f_n) \in \QQ(x_1 ,\ldots, x_m)^n$;
\Ensure $b=$ the answer to ``is $I$ binomial after a linear coordinate
change?'';
\State ${{\mathcal{B}}} \gets (A_1 ,\ldots, A_l)$, a basis of $L$, the output of
\Cref{alg:Symmetry2} 
\State $A \gets \sum_i y_i A_i$; \Comment{The $y_i$ are variables.}
\State $A_s \gets $ semisimple part of $A$; 
\State ${{\mathcal{C}}} \gets $ a basis over $\QQ(y_1 ,\ldots, y_l)$ of 
$\{D \in \QQ(y_1 ,\ldots, y_l) \otimes L \mid [A_s,D]=0\}$;
\For{$D \in {{\mathcal{C}}}$}
	\State $D \gets $ semisimple part of $D$;
\EndFor;
\State $p \gets (x_1 ,\ldots, x_n)$;
\State $q \gets \phi(p)$;
\State $J \gets (\frac{\partial f_i}{\partial x_j})(p)$;
\State $b \gets [$all columns of $J$ are in the
$\CC(x_1 ,\ldots, x_m,y_1 ,\ldots, y_l)$-span of $\{Dq \mid D \in
{{\mathcal{C}}}\}]$;
\State \Return $b$;
\end{algorithmic}
\end{algorithm}

\begin{prop}
\Cref{alg:Binom2} gives the correct output.
\end{prop}

\begin{proof}
We need to check whether for sufficiently general rational numbers
$t_1 ,\ldots, t_l$, the linear combination $M=\sum_i t_i A_i$ has the
property that the semisimple part of the centralizer of~$M_s \in
\frakg_\QQ$, acting
on a sufficiently general point $q \in X$, yields a space containing
(and hence equal to) the tangent space $T_q X$. This is done in the
algorithm; we use here that for $t_1 ,\ldots, t_{{\ell}}$ sufficiently general,
the semisimple part over $\QQ(y_1 ,\ldots, y_l)$ of the {\em symbolic}
matrix $A$ specializes to the semisimple part of $M$ by substituting
$y_i \mapsto t_i$.
\end{proof}

\section{Computational experiments}
\label{sec:computational experiments}
A recurring theme in algebraic statistics is to determine when a statistical model is $\GL$-binomial, that is, it can be realized as the solution set of binomial equations, possibly after a linear change of variables. One then uses these binomials, together with monomial parametrizations of the variety and the geometry of the associated polytope, to address statistical questions, particularly in model selection and the study of likelihood geometry. In this spirit, we apply \Cref{alg:Binom} to two families of models defined by rational parametrizations: staged tree models and colored Gaussian graphical models.

\subsection*{Staged tree models}
Staged tree models are discrete models that represent dependency relations among events and are realized in a staged tree $\mathcal T$, that is, a rooted directed tree with colored nodes and labeled edges. Nodes represent events, and edge labels are conditional probabilities between events.  Let the stages or colors in the internal nodes of $\mathcal T$ be  $1 ,\ldots, S$. Nodes of the same color $s\in [S]$, must have the same number of outgoing edges $m_s$ and the same labels $\theta_{s,1} ,\ldots, \theta_{s,m_s}$ in their outgoing edges. {{Given that the labels represent probabilities, the last parameter is an expression in terms of the other parameters, that is, $\theta_{s,m_s}=1-(\theta_{s,1}+\ldots +\theta_{s,m_s-1})$. 
  So, the dimension of the model is $\prod_{i=1}^S(m_i-1)$. Let the leaves of $\mathcal T$ be $1 ,\ldots,  n$. The leaves of the tree represent the final events. Their  probability is the product of the edge labels along the path from the root to the leaf.  Let $\lambda(\ell)$ be the sequence of edges on the path from root to leaf~$\ell$,  and $\theta(e)$ denote the label for edge $e$. The depth of  $\mathcal T$  is  the maximum among all $\#\lambda(\ell)$. 
In algebraic terms, the staged tree model for $\mathcal T$ is the image of parametrization
 \begin{align}\label{eq:parametrization_for-staged_trees}
 \varphi\colon \CC^{m_1+\cdots+m_S-S+1} \dto  \CC^{n}, \quad (\theta_{s,i} \mid s=1 ,\ldots,  S, i=1,\ldots m_s-1)  \mapsto (\prod\limits_{e\in \lambda(\ell)}\theta(e))_{\ell=1}^n.
 \end{align}
 For illustration, consider the first binary staged tree in \Cref{table:binary_two_staged} which has two stages, stage one red and stage two green. The model has parametrization  given by polynomials:
\begin{gather*}
    \begin{aligned}
        f_1 &= \theta_{1,1} \theta_{2,1} \\
        f_4 &= (1-\theta_{1,1}) \theta_{1,1} \\
    \end{aligned} \quad 
    \begin{aligned}
       f_2 &= \theta_{1,1} (1- \theta_{2,1}) \theta_{2,1}\\
        f_5 &= (1-\theta_{1,1}) (1-\theta_{1,1}) \theta_{2,1} \\
    \end{aligned} \quad 
     \begin{aligned}
        f_3 &= \theta_{1,1} (1- \theta_{2,1}) (1-\theta_{2,1})\\
        f_6 &= (1-\theta_{1,1}) (1-\theta_{1,1}) (1-\theta_{2,1}).\\
    \end{aligned}
\end{gather*}
By \cite[Theorem~6.5]{gorgen2022staged}, all binary one-stage tree models, that is when $S=1$, are $\GL$-binomial given by Veronese embeddings. Ternary one-stage tree models of depth at most $3$ also are $\GL$-binomial by \cite[Theorem~6.7]{gorgen2022staged}, while in depth $4$, the caterpillar tree $\mathcal C_{3,4}$, the first tree in \Cref{table:caterpillar2}, is not $\GL$-binomial by \cite[Example~26]{Maraj23}. Binary staged trees are not all $\GL$-binomial either: an example with $4$ stages is given in \cite[Example~25]{Maraj23} and one with $10$ stages in \cite[Theorem~3.6]{nicklasson2023toric}.

To understand when $\GL$-binomiality fails, we undertake a systematic study of binary tree models in two stages. We find that non-$\GL$-binomial staged tree models already occur for trees of depth as low as~$3$.

\begin{ex}[Binary two-stage tree models]\label{ex:binary two-stage trees}
We tested $\GL$-binomiality for all $127$ two-stage binary tree models with depth at most $3$, and found that exactly $4$ are not $\GL$-binomial. These are listed in \Cref{table:binary_two_staged} together with the dimension of their Lie algebra. Labeling leaves in increasing order from top to bottom, the last $2$ models, with Lie algebras of dimension~$9$ and~$10$, lie in the hypersurface $y_2=y_3$ (since $f_2=f_3$). As noted in \Cref{re: linearly degenerate}, it is wise to reduce to a smaller ambient dimension in this case. Removing $f_3$ from $\phi$ reduces the symmetry Lie algebra to dimension $2$ for both models. Bases for all $4$ models are given in \Cref{table:binary_two_staged}. In each case, the maximal torus has dimension $1$, which is less than the dimension $2$ of the model.
\end{ex}
\begin{table}[!ht]
\centering
\begin{tabular}{ c | c | c }
staged tree model $X$ & $\dim\mathfrak{g}_X$ & basis for $\mathfrak{g}_{\Tilde{X}}$ \\ \hline
\parbox{2.5cm}{\includegraphics[width=2.5cm]{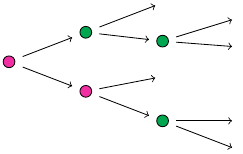}} & 2 &
\resizebox{!}{1cm}{$
\begin{bmatrix}
1 & 0 & 0 & 0 & 0 & 0 \\
-1 & 2 & 0 & 0 & 0 & 0 \\
0 & -2 & 0 & 0 & 0 & 0 \\
0 & 0 & 0 & 0 & 0 & 0 \\
0 & 0 & 0 & 0 & 1 & 0 \\
0 & 0 & 0 & 0 & -1 & 0 \\
\end{bmatrix}
\quad
\begin{bmatrix}
0 & 1 & 1 & 0 & 0 & 0 \\
0 & -1 & 1 & 0 & 0 & 0 \\
0 & 0 & -2 & 0 & 0 & 0 \\
0 & 0 & 0 & 0 & 0 & 0 \\
0 & 0 & 0 & 0 & 0 & 1 \\
0 & 0 & 0 & 0 & 0 & -1 \\
\end{bmatrix}$}
\\ & & \\
\parbox{2.5cm}{\includegraphics[width=2.5cm]{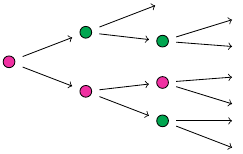}} & 2 &
\resizebox{!}{1cm}{$
\begin{bmatrix}
1 & 0 & 0 & 0 & 0 & 0 & 0 \\
-1 & 2 & 0 & 0 & 0 & 0 & 0 \\
0 & -2 & 0 & 0 & 0 & 0 & 0 \\
0 & 0 & 0 & 0 & 0 & 0 & 0 \\
0 & 0 & 0 & 0 & 0 & 0 & 0 \\
0 & 0 & 0 & 0 & 0 & 1 & 0 \\
0 & 0 & 0 & 0 & 0 & -1 & 0 \\
\end{bmatrix}
\quad
\begin{bmatrix}
0 & 1 & 1 & 0 & 0 & 0 & 0 \\
0 & -1 & 1 & 0 & 0 & 0 & 0 \\
0 & 0 & -2 & 0 & 0 & 0 & 0 \\
0 & 0 & 0 & 0 & 0 & 0 & 0 \\
0 & 0 & 0 & 0 & 0 & 0 & 0 \\
0 & 0 & 0 & 0 & 0 & 0 & 1 \\
0 & 0 & 0 & 0 & 0 & 0 & -1 \\
\end{bmatrix}$}
\\ & & \\
\parbox{2.5cm}{\includegraphics[width=2.5cm]{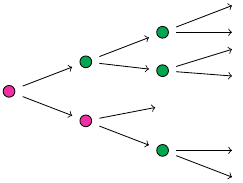}} & 9 &
\resizebox{!}{1cm}{$
\begin{bmatrix}
2 & 0 & 0 & 0 & 0 & 0 \\
-1 & 1 & 0 & 0 & 0 & 0 \\
0 & -2 & 0 & 0 & 0 & 0 \\
0 & 0 & 0 & 0 & 0 & 0 \\
0 & 0 & 0 & 0 & 1 & 0 \\
0 & 0 & 0 & 0 & -1 & 0 \\
\end{bmatrix}
\quad
\begin{bmatrix}
0 & 2 & 0 & 0 & 0 & 0 \\
0 & -1 & 1 & 0 & 0 & 0 \\
0 & 0 & -2 & 0 & 0 & 0 \\
0 & 0 & 0 & 0 & 0 & 0 \\
0 & 0 & 0 & 0 & 0 & 1 \\
0 & 0 & 0 & 0 & 0 & -1 \\
\end{bmatrix}$}
\\ & & \\
\parbox{2.5cm}{\includegraphics[width=2.5cm]{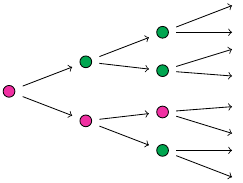}} & 10 &
\resizebox{!}{1cm}{$
\begin{bmatrix}
2 & 0 & 0 & 0 & 0 & 0 & 0 \\
-1 & 1 & 0 & 0 & 0 & 0 & 0 \\
0 & -2 & 0 & 0 & 0 & 0 & 0 \\
0 & 0 & 0 & 0 & 0 & 0 & 0 \\
0 & 0 & 0 & 0 & 0 & 0 & 0 \\
0 & 0 & 0 & 0 & 0 & 1 & 0 \\
0 & 0 & 0 & 0 & 0 & -1 & 0 \\
\end{bmatrix}
\quad
\begin{bmatrix}
0 & 2 & 0 & 0 & 0 & 0 & 0 \\
0 & -1 & 1 & 0 & 0 & 0 & 0 \\
0 & 0 & -2 & 0 & 0 & 0 & 0 \\
0 & 0 & 0 & 0 & 0 & 0 & 0 \\
0 & 0 & 0 & 0 & 0 & 0 & 0 \\
0 & 0 & 0 & 0 & 0 & 0 & 1 \\
0 & 0 & 0 & 0 & 0 & 0 & -1 \\
\end{bmatrix}$}
\\
\end{tabular}
\caption{All binary two-stage tree models with depth at most $3$ that are
not $\GL$-binomial, the dimension of their symmetry Lie algebras, and a basis for the symmetry  Lie algebra of the model in the smaller ambient dimension. These models have all dimension $2$ and maximal torus of dimension $1$.}
\label{table:binary_two_staged}
\end{table}

To better understand one-stage tree models, we also consider extensions of the caterpillar tree model $\mathcal{C}_{3,4}$, the first tree in \Cref{table:caterpillar2}, the smallest  one-stage tree model which is  not $\GL$-binomial. Since the set of all ternary one-stage trees is quite large, we restrict to extensions of $\mathcal{C}_{3,4}$ by at most $3$ internal nodes within the set of ternary one-stage trees of depth $4$.

\begin{ex}[Ternary one-stage trees of depth $4$]\label{ex:caterpillar}
Consider extensions of $\mathcal{C}_{3,4}$ to ternary one-stage trees obtained by iteratively changing up to $3$ leaves into internal nodes, with the condition that the new tree still has depth $4$. We learn that $4$ out of $6$ extensions by $1$ internal node are not $\GL$-binomial, $13$ out of $27$ extensions by $2$ internal nodes are not $\GL$-binomial, and $35$ out of $110$ extensions by $3$ internal nodes are not $\GL$-binomial. \Cref{table:caterpillar2} lists extensions of $\mathcal{C}_{3,4}$ by at most $2$ internal nodes whose model is not $\GL$-binomial, together with the dimensions of their symmetry Lie algebra. For the not $\GL$-binomial extensions by $3$ internal nodes the dimension of the symmetry Lie algebra ranges from $49$ to $78$. These dimensions are large because the models are linearly degenerate. In each case, the maximal torus has dimension $1$, which is less than the dimension $2$ of the model.
\end{ex}

\begin{table}[!ht]
\centering
\begin{tabular}{c  c | c  c | c c}
staged tree model $X$ & $\dim \mathfrak{g}_X$ & staged tree model $X$  & $\dim \mathfrak{g}_X$ & staged tree model $X$  & $\dim \mathfrak{g}_X $\\
\parbox{2.2cm}{\includegraphics[width=2.2cm]{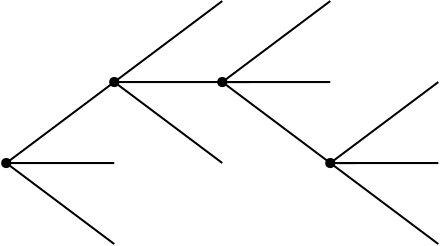}} & 1 &
\parbox{2.2cm}{\includegraphics[width=2.2cm]{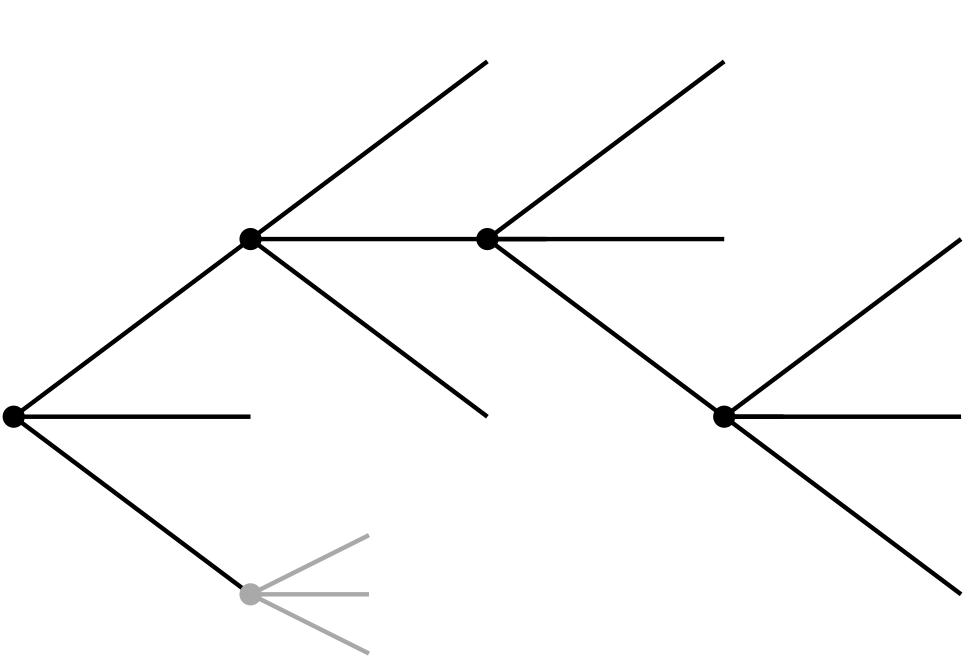}}  & 12 &
\parbox{2.2cm}{\includegraphics[width=2.2cm]{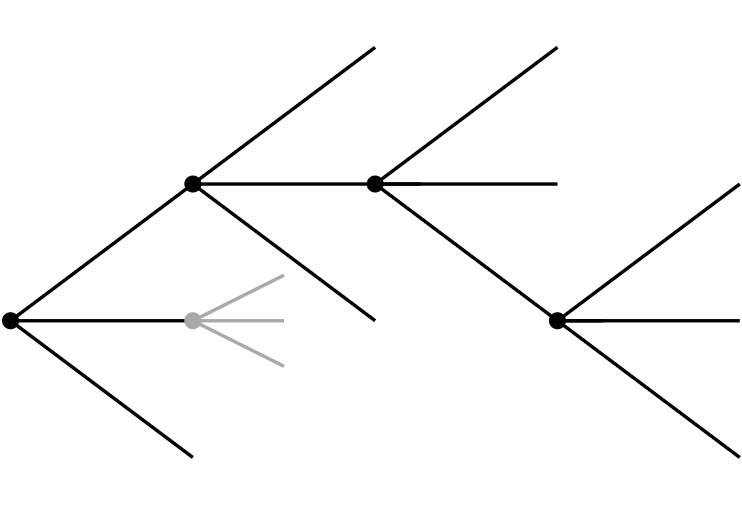}} & 12 \\
\parbox{2.2cm}{\includegraphics[width=2.2cm]{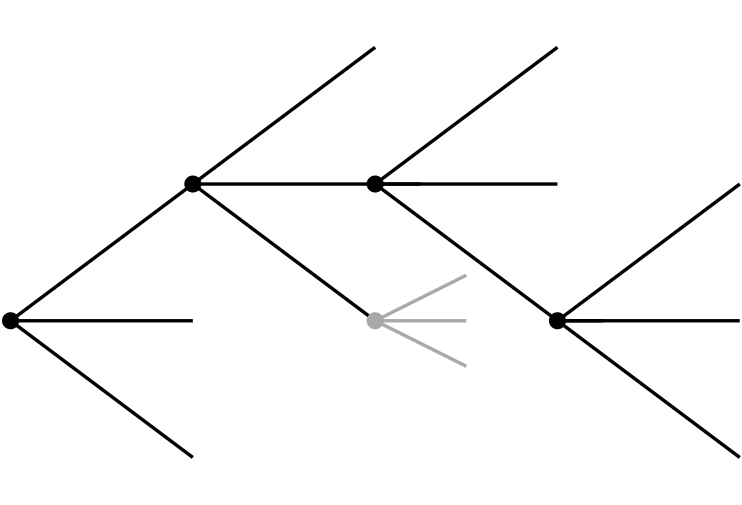}} &13 &
\parbox{2.2cm}{\includegraphics[width=2.2cm]{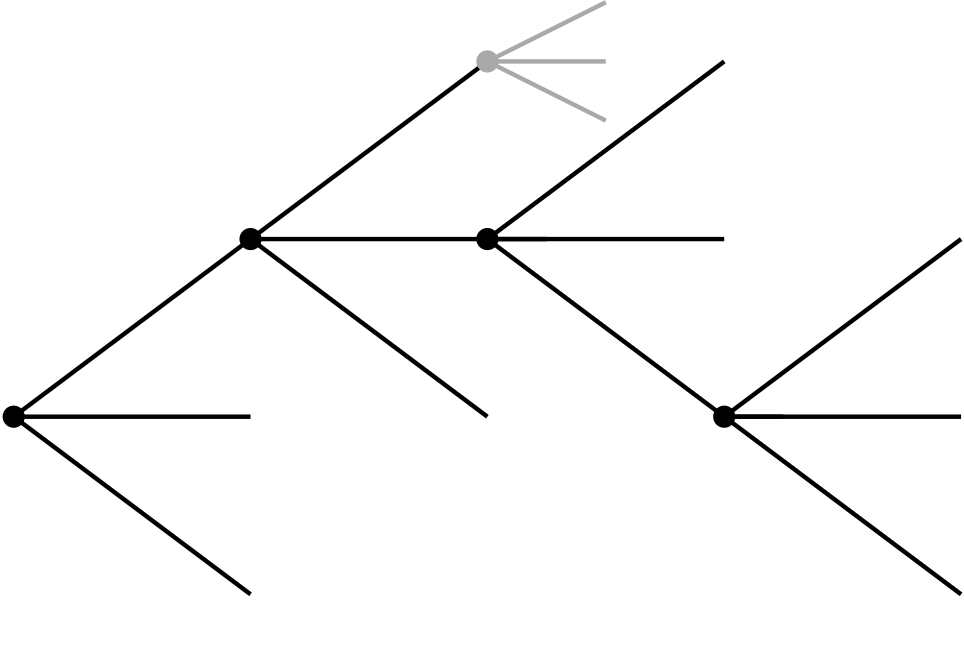}} & 13 &
\parbox{2.5cm}{\includegraphics[width=2.3cm]{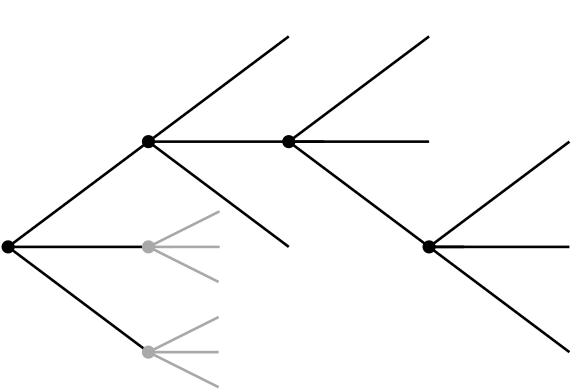}} & 40 \\
\parbox{2.5cm}{\includegraphics[width=2.3cm]{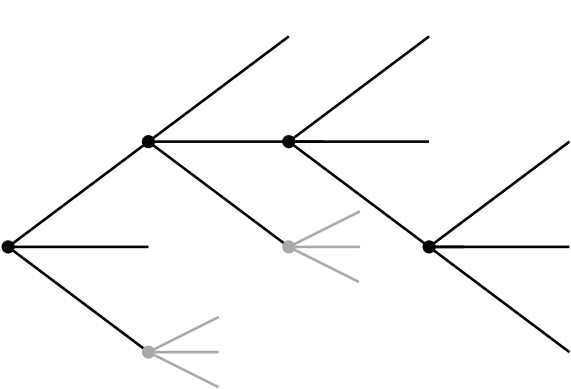}} & 28 & 
\parbox{2.5cm}{\includegraphics[width=2.3cm]{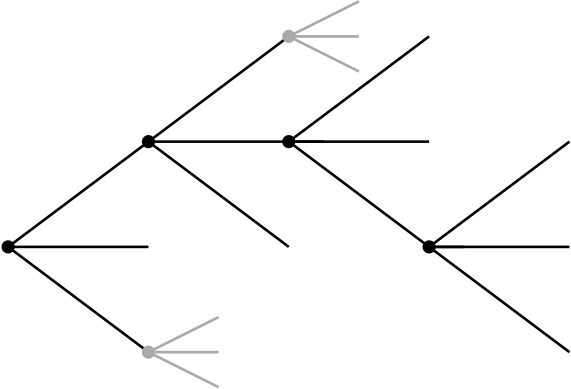}} & 28 &
\parbox{2.5cm}{\includegraphics[width=2.3cm]{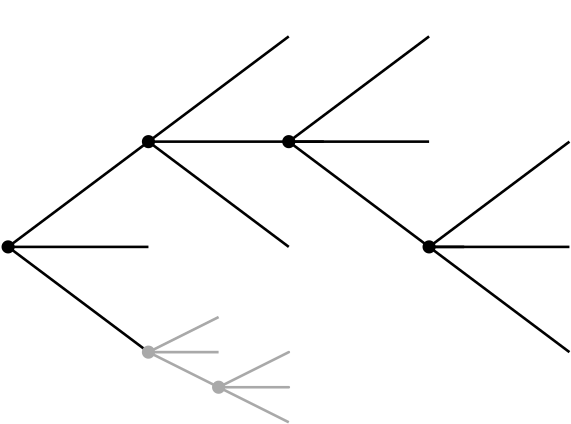}} & 16 \\
\parbox{2.5cm}{\includegraphics[width=2.3cm]{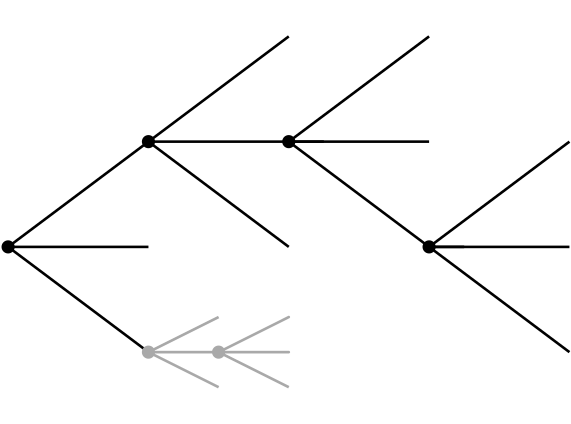}} & 28 &
\parbox{2.5cm}{\includegraphics[width=2.3cm]{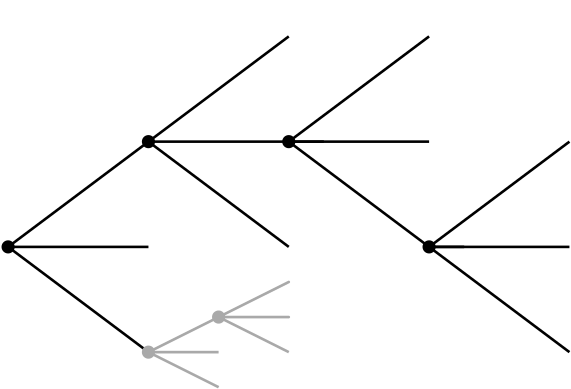}} & 28&
\parbox{2.5cm}{\includegraphics[width=2.3cm]{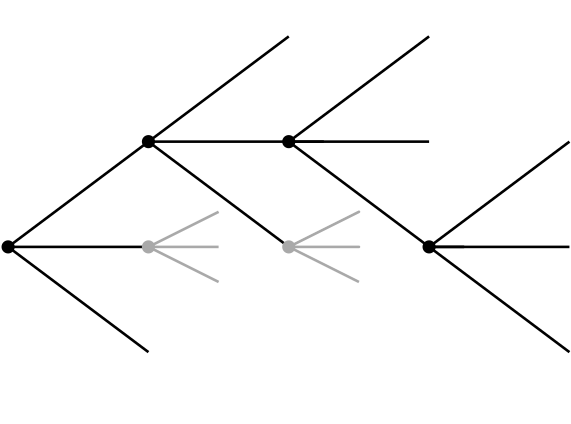}} & 28\\ 
\parbox{2.5cm}{\includegraphics[width=2.3cm]{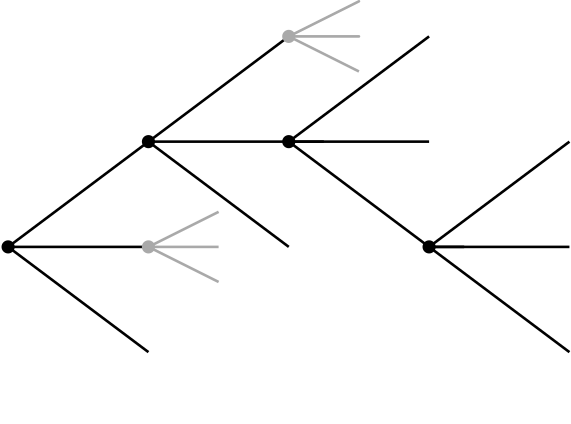}} & 28 & 
\parbox{2.5cm}{\includegraphics[width=2.3cm]{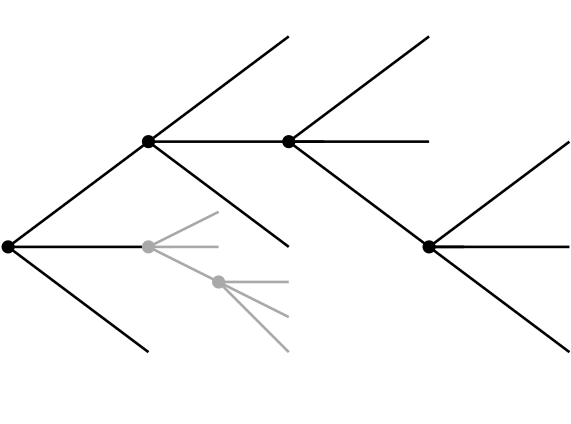}} & 28&
\parbox{2.5cm}{\includegraphics[width=2.3cm]{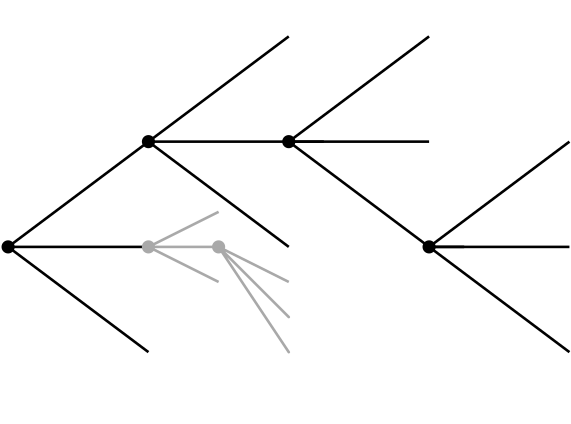}} & 28 \\
\parbox{2.5cm}{\includegraphics[width=2.3cm]{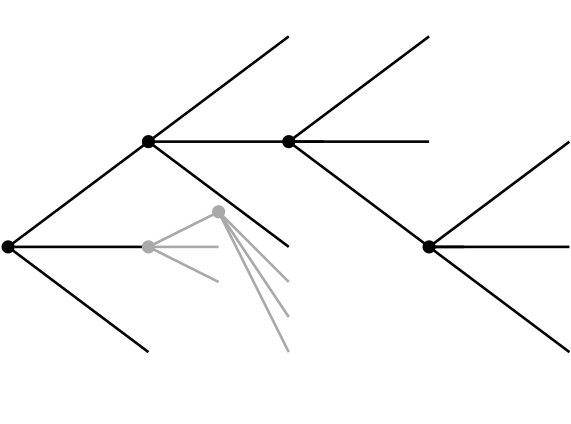}} & 40 &
\parbox{2.5cm}{\includegraphics[width=2.3cm]{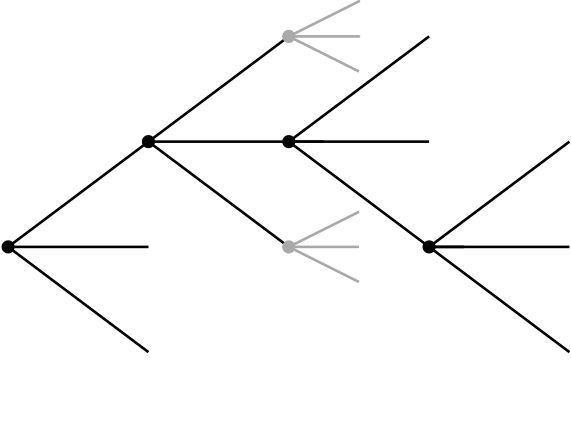}} & 41&
\parbox{2.5cm}{\includegraphics[width=2.3cm]{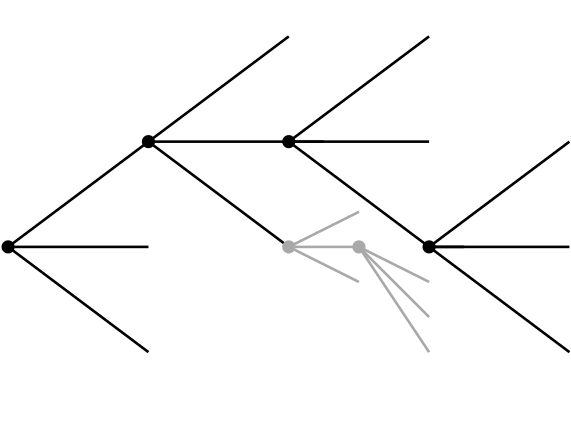}} & 41\\
\end{tabular}
\caption{All extensions of the caterpillar tree $\mathcal{C}_{3,4}$ by at most $2$ internal nodes that give non-$\GL$-binomial one-stage tree models, and the dimensions of their Lie algebras.
}
\label{table:caterpillar2}
\end{table}    
}}

\subsection*{Colored Gaussian graphical models}
Colored Gaussian graphical models are linear concentration models used to model dependency relations among i.i.d. Gaussian random variables, encoded in a colored graph. The nodes represent the variables, a missing edge implies that the associated entry in the concentration matrix $K$ must be zero, and nodes and edges that share a color imply that the corresponding entries in $K$ must be the same. These models generalize Gaussian graphical models, which are defined solely by the missing-edge conditions. 
More precisely, let $G$ be a colored graph on the vertex set $\{1 ,\ldots, n\}$, color set $\{1 ,\ldots,  m\}$ and edge set $E$, where edges and vertices do not share colors. Let  $\theta_1 ,\ldots, \theta_m$ be parameters associated to each color. Consider the $m$-dimensional linear space of $n\times n$ symmetric matrices $\mathcal{L}_G\cong \CC^m$ with constraints for $K \in \CC^{n \times n}$ given by 
\[
 K\in \mathcal{L}_G \quad \text{ if and only if } \quad
 \begin{array}{l}
  k_{ii} = \theta_{\mathrm{color}(i)} \text{ for } 1\leq i\leq n,\\
 k_{ij} = 0 \text{ if } i \neq j \text{ and } \{i,j\}\not \in E, \\
 k_{ij} = \theta_{\mathrm{color}(\{i,j\})} \text{ if } \{i,j\} \in E.
 \end{array}
\]
The reciprocal variety $\calL_G^{-1}$ of $\calL_G$ is the Zariski closure of the image of the rational map
\begin{align}\label{eq:colored graph}
\phi_G: \CC^{m} \dto \CC^{\binom{n+1}{2}},\ (\theta_i)_{i=1}^m \mapsto ((K(\theta)^{-1})_{ij})_{1\leq i\leq j\leq n}. 
\end{align}
This variety can  equivalently be given by the image of the map defined by the maximal minors of $K(\theta)$. The intersection of $\calL_G^{-1}$ with the cone of positive definite matrices $PD_n$ determines the set of covariance matrices of the colored Gaussian graphical model. We say that the colored Gaussian graphical model for $G$ is $\GL$-binomial if $\calL_G^{-1}$ is $\GL$-binomial. 

{{Reciprocal varieties of linear spaces of matrices have been studied in different contexts, both in algebraic statistics in \cite{cardwell2024toric,sturmfels2010multivariate,sturmfels2019brownian}, and in other areas of mathematics such as Schubert calculus \cite{MMMSV}, and the study of tensors \cite{ConMic21,gesmundo2025collineation}.

The key takeaway of the experiments in \cite[Section~6]{Kahle25} is that a Gaussian graphical model is $\GL$-binomial if and only if it is binomial in the original variables, which occurs if and only if the graph is a block graph. This does not hold for colored graphs: \cite[Theorem~1]{biaggi2025binomiality} shows that \emph{triangle-regular block graphs}, that is, block graphs that are vertex-regular, edge-regular, and vertex-triangle-regular, are exactly those graphs for which $I(\mathcal{L}^{-1}_{G})$ is binomial in the original variables. Furthermore, \cite{cardwell2024toric} shows that there are colored Gaussian graphical models on block graphs outside this class that are $\GL$-binomial.

It is natural to ask whether all colorings of block graphs give a binomial model, and whether certain colorings of non-block graphs can also give binomial models. To investigate this, we tested $\GL$-binomiality for all colorings of the complete graph $K_4$ and the cycle $C_4$ on $4$ nodes. We found that most colorings of block graphs give non-$\GL$-binomial models, while some colorings of non-block graphs do yield $\GL$-binomial models.

\begin{ex}[Colorings of complete graph $K_4$]\label{ex:K4} The reciprocal variety $\mathcal{L}^{-1}_{K_4}$ is the whole space of $n\times n$ symmetric matrices, and hence it is $\GL$-binomial. We computed the symmetry Lie algebras and tested $\GL$-binomiality for all $215$ colorings of $K_4$. Only $40$ of them yield binomial models, with dimension of the Lie algebra ranging from $5$ to $67$. In \Cref{fig:K4} we list the colorings of $K_4$ that are not triangle-regular, but whose model is  binomial, together with the dimension of the Lie algebra of each model after linear relations are removed. The models that are not $\GL$-binomial have Lie algebras of dimension ranging from $1$ to $52$.

\begin{table}[!ht]
\centering
\begin{minipage}{0.7\textwidth}
\centering
\resizebox{\textwidth}{!}{%
\begin{tabular}{ c c c c c c}
 \begin{minipage}{2.3cm}
\centering
\includegraphics[width=2.3cm]{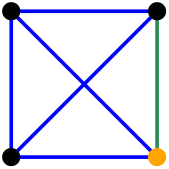}\\
$4 \mid 38 \mid 8$
\end{minipage} &
\begin{minipage}{2.3cm}
\centering
\includegraphics[width=2.3cm]{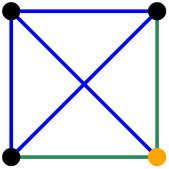}\\
$4 \mid 38 \mid 8$
\end{minipage} &
\begin{minipage}{2.3cm}
\centering
\includegraphics[width=2.3cm]{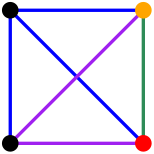}\\
$6 \mid 18 \mid 18$
\end{minipage} &
\begin{minipage}{2.3cm}
\centering
\includegraphics[width=2.3cm]{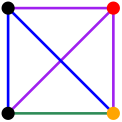}\\
$6 \mid 18 \mid 18$
\end{minipage} &
\begin{minipage}{2.3cm}
\centering
\includegraphics[width=2.3cm]{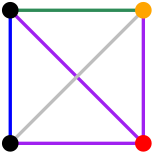}\\
$7 \mid 33 \mid 33$
\end{minipage} &
\begin{minipage}{2.3cm}
\centering
\includegraphics[width=2.3cm]{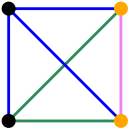}\\
$5 \mid 42 \mid 12$
\end{minipage}
\end{tabular}
}
(a)
\end{minipage}
\begin{minipage}{0.2\textwidth}
\centering
\includegraphics[width=2cm]{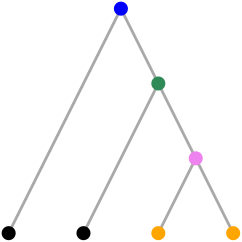}\\
(b)
\end{minipage}
\caption{In (a) are all not vertex regular, not edge regular and not vertex triangular colorings of $K_4$ with $\GL$-binomial Gaussian graphical model. The numbers below each colored graph are the dimension of the model, the dimension of its symmetry Lie algebra, and the dimension of the symmetry Lie algebra after removing linear relations. Graph $6$ is a BMT-derived graph with phylogenetic tree (b). }
\label{fig:K4}
\end{table}
\end{ex}

BMT-derived graphs are colored graphs whose coloring is induced by a phylogenetic tree with colored and zeroed nodes, introduced in \cite{cardwell2024toric}. More precisely, $k_{ii}$ is the parameter corresponding to the color of leaf $i$ in the colored tree, $k_{ij}=0$ when the internal node corresponding to the least common ancestor of $i$ and $j$ is labeled $0$, and $k_{ij}=k_{uv}$ when leaves~$i,j$ and $u,v$ share the same least common ancestor. The paper proves that all vertex-regular BMT-derived graphical models are $\GL$-binomial, and conjectures that the converse also holds. We provide a counterexample.

\begin{ex}[A counterexample to {\cite[Conjecture 22]{cardwell2024toric}}] \label{ex: conjecture}
The colored graph $G_6$ in \Cref{fig:K4}(a) is a BMT-derived graph, derived from the colored phylogenetic tree in \Cref{fig:K4}(b).
Let $X=\overline{\im \phi_{G_6}}$. We have  
\begin{align*}
  I(X)=\langle &\sigma_{33}-\sigma_{44}, \sigma_{23}-\sigma_{24},\sigma_{13}-\sigma_{14}, 2\sigma_{14}^2 -2\sigma_{24}^2-\sigma_{11}\sigma_{34}+\sigma_{22}\sigma_{34}-\sigma_{11}\sigma_{44}+\sigma_{22}\sigma_{44}, \\
  & \sigma_{12}\sigma_{14}-\sigma_{11}\sigma_{24}-2\sigma_{14}\sigma_{24}+\sigma_{12}\sigma_{34}+\sigma_{12}\sigma_{44}\rangle, 
\end{align*}
which under the linear change of variables
\begin{align*}
\begin{bmatrix}
\sigma_{11} \\
\sigma_{12} \\
\sigma_{22} \\
\sigma_{13} \\
\sigma_{23} \\
\sigma_{33} \\
\sigma_{14} \\
\sigma_{24} \\
\sigma_{34} \\
\sigma_{44} \\
\end{bmatrix}\mapsto\underbrace{\begin{bmatrix}
12 & 15 & 60 &0 & 0 & 0& 15 & 0 & 60 & 0 \\
0 & -9 & -60 &0 & 0 & 0& 9 & 0 & 60 & 0 \\
0 & 3 & 60&0 & 0 & 0 & 3 & 12 & 60 & 0 \\
-6 & -3 & 0 & -3 & 0 & 0 & 0 &0 & 0 & 0\\
0 & 3 & 0 & 0 & -3 & 0 & 0 &0 & 0 & 0 \\
6 & 12 & 40 & 0 & 0 & -3 & 0 &0 & 0 & 0\\
-6 & -3 & 0 &0 & 0 & 0& -3 & 0 & 0 & 0 \\
0 & 3 & 0 &0 & 0 & 0& -3 & 0 & 0 & 0 \\
0 & -12 & -40&0 & 0 & 0 & 0 & 0 & 0 & 3 \\
6 & 12 & 40 &0 & 0 & 0 & 0 & 0 & 0 & -3
\end{bmatrix}}_{g}
\begin{bmatrix}
\sigma_{11} \\
\sigma_{12} \\
\sigma_{22} \\
\sigma_{13} \\
\sigma_{23} \\
\sigma_{33} \\
\sigma_{14} \\
\sigma_{24} \\
\sigma_{34} \\
\sigma_{44} \\
\end{bmatrix}
\end{align*}
becomes the binomial ideal
\begin{align*}
 gI(X)=\langle &\sigma_{33}-\sigma_{44}, \sigma_{23}-\sigma_{24},\sigma_{13}-\sigma_{14}, \sigma_{22}\sigma_{14}-\sigma_{12}\sigma_{34},\sigma_{12}\sigma_{14}-\sigma_{11}\sigma_{24}\rangle.
\end{align*} 
\end{ex}

\begin{ex}[Colorings of the cycle $C_4$]\label{ex:C4} The Gaussian graphical model on $C_4$ is not $\GL$-binomial.  However,  $8$ out of $74$ colorings of $C_4$, shown in \Cref{fig:C4}, give colored Gaussian graphical models that are $\GL$-binomial. None of them are binomial in the original variables, i.e., in all cases a change of variables by a non-diagonal matrix $g$ is required such that $g\mathcal{L}^{-1}_G$ is generated by binomials.  
\end{ex}

\begin{table}[!ht]
\centering
\begin{tabular}{c c c c c c c c}
\begin{minipage}{1.62cm}
\centering
\includegraphics[width=1.62cm]{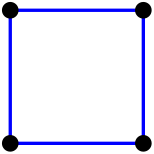}\\
$2 \mid 74 \mid 4$
\end{minipage} &
\begin{minipage}{1.62cm}
\centering
\includegraphics[width=1.62cm]{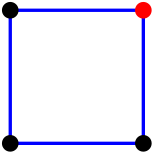}\\
$3 \mid 47 \mid 7$
\end{minipage} &
\begin{minipage}{1.62cm}
\centering
\includegraphics[width=1.62cm]{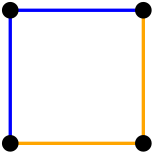}\\
$3 \mid 49 \mid 9$
\end{minipage} &
\begin{minipage}{1.62cm}
\centering
\includegraphics[width=1.62cm]{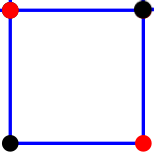}\\
$3 \mid 53 \mid 3$
\end{minipage} &
\begin{minipage}{1.62cm}
\centering
\includegraphics[width=1.62cm]{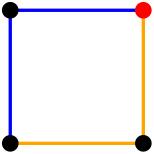}\\
$4 \mid 7 \mid 7$
\end{minipage} &
\begin{minipage}{1.62cm}
\centering
\includegraphics[width=1.62cm]{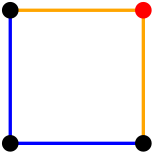}\\
$4 \mid 34 \mid 4$
\end{minipage} &
\begin{minipage}{1.62cm}
\centering
\includegraphics[width=1.62cm]{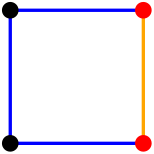}\\
$4 \mid 52 \mid 12$
\end{minipage} &
\begin{minipage}{1.62cm}
\centering
\includegraphics[width=1.62cm]{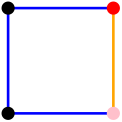}\\
$5 \mid 29 \mid 19$
\end{minipage}
\end{tabular}
\caption{All colorings of $C_4$ with $\GL$-binomial colored Gaussian graphical model. Below each graph is the dimension of the model, the dimension of its symmetry Lie algebra, and of the symmetry Lie algebra after removing linear relations.}
\label{fig:C4}
\end{table}

Finally, we record the fact that colored Gaussian graphical models with the minimal trivial coloring, that is, with $2$ colors, one for the vertices and one for the edges, are all $\GL$-binomial. To see this, we realize $\mathcal{L}^{-1}_G$ as the affine cone over the $(n-1)$-th collineation variety of a pencil of matrices, that is, the closure of the image of the rational map sending a matrix to its tuple of $(n-1)\times (n-1)$ minors. By \cite[Theorem~3.1]{gesmundo2025collineation}, collineation varieties of matrix pencils are rational normal curves whose degree is controlled by the intersection of the pencil with the varieties of matrices of bounded rank; in particular, they are $\GL$-binomial.
\begin{thm}\label{thm: two colors}
Let $G$ be a colored graph where all vertices share the same color and all edges share the same color. The reciprocal variety $\mathcal{L}^{-1}_G$ of the associated colored Gaussian graphical model is the affine cone over a rational normal curve. Its degree is bounded above by the number of distinct eigenvalues of the adjacency matrix of $G$.
\end{thm}
\begin{proof}
Let $\theta_1$ be the parameter for the vertex color and $\theta_2$ the parameter for the edge color, and let $n$ be the number of vertices of $G$. Consider the matrix pencil parametrized by
\[
K_G\colon (\theta_1,\theta_2) \mapsto \theta_1 \Id_{n} + \theta_2 A_G,
\]
where $A_G$ is the adjacency matrix of $G$. Following the definitions of \cite{gesmundo2025collineation}, the reciprocal variety of $\mathcal{L}_G$ is the cone over the $(n-1)$-th collineation variety $\mathscr{C}^{(n-1)}(K_G)\subseteq \PP\CC^{^{\binom{n+1}{2}}}$.  So \cite[Theorem~3.1]{gesmundo2025collineation} guarantees that $\mathscr{C}^{(n-1)}(K_G)$ is a rational normal curve, and therefore~$\mathcal{L}_G^{-1}$ is a cone over a rational normal curve.

The degree of $\mathscr{C}^{(n-1)}(K_G)$ is $n$ minus the sum of the intersection multiplicities between the projective line $\mathbb{P} \mathcal{L}_G \subseteq \mathbb{P} (\mathbb{C}^n \otimes \mathbb{C}^n)$ and the variety $\sigma_{n-2}$ of symmetric matrices of size $n$ and rank at most $n-2$. Notice that $K(\lambda,1)$ is not full rank if and only if $\lambda$ is an eigenvalue of $A_G$. Moreover, the intersection multiplicity $\mathrm{imult}_{K(\lambda,1)} ( \mathbb{P} \mathcal{L}_G , \sigma_{n-2})$ at $(\lambda,1)$ is bounded below by $\mathrm{amult}(\lambda) -1$, where $\mathrm{amult}(\lambda)$ is the algebraic multiplicity of $\lambda$ as an eigenvalue of $A_G$.

If $\lambda_1 ,\ldots, \lambda_s$ are the distinct eigenvalues of $A_G$, we deduce that
\begin{align*}
\deg \mathcal{L}_G^{-1} =& \deg \mathscr{C}^{(n-1)}(K_G) = n - \textstyle \sum\limits_{i=1}^s \mathrm{imult}_{K(\lambda_i,1)} ( \mathbb{P} \mathcal{L}_G , \sigma_{n-2}) \leq n - \textstyle \sum\limits_{i=1}^s(\mathrm{amult}(\lambda_i) - 1) = s,
\end{align*}
where in the last equality we use that the sum of algebraic multiplicities equals $n$.
\end{proof}

\Cref{thm: two colors} implies in particular that all colored Gaussian graphical models on cycles with~$2$ colors are $\GL$-binomial. This also occurs for other colorings, as shown in \Cref{ex:C4}. It would be interesting to use \Cref{alg:Binom} to classify all colorings of the cycle $C_n$ for $n\geq 5$ that yield $\GL$-binomial models. Since most of these models are linearly degenerate, the analysis can be simplified by using results on their linear relations \cite{davies2021coloured,gobel2025linear}.
}}

\section{Symmetry Lie algebras of curves} \label{sec:Curves}

In this section, we use \Cref{thm: frakg via tangent spaces} and material from \Cref{sec:Binomial} to characterize when a rational curve $X$ given by a parametrization has $\frakg_X \neq \{0\}$.

We first establish some terminology. The {\em leading coefficient}
of a nonzero rational function~{{$r=a/b$ for polynomials $a,b \in \CC[t]$}}
is by definition the leading coefficient of $a$ divided by that of $b$. If
the leading coefficient of $r$ equals $1$, we call $r$ {\em monic}. The
{\em degree} of $r$ equals $\deg(a)-\deg(b) \in \ZZ$, the order of the
pole at infinity, and we set $\deg(0) \coloneq -\infty$. These notions behave
similarly to the case of polynomials. In particular, if $r,s$ are rational
functions of different degrees, then $\deg(r+s)=\max\{\deg(r),\deg(s)\}$
and if $r,s$ are monic rational functions of the same degree, then
$\deg(r-s)<\max\{\deg(r),\deg(s)\}$. Also, if $r$ is a monic rational
function of degree $d$, then its derivative $r'$ has degree $d-1$ and
leading coefficient $d$, unless $d=0$, in which case $\deg(r')$ is a
value between $-\infty$ (if $r$ is a constant) and $d-2$. 

If $X$ is parametrized by Laurent monomials $t^{e_i}$, then $\frakg_X$ is easy to~characterize.
{{
\begin{lm}\label{lem:monomial curves}
Let $\phi = (t^{e_1} ,\ldots, t^{e_n} )\colon  \CC \dto \CC^n$ be a rational map with $e_1 < e_2 < \cdots < e_n$ and $\gcd(e_1 ,\ldots, e_n) = 1$, and let  $X = \overline{\im(\phi)}$. Let
 $D \coloneq \diag(e_1 ,\ldots, e_n)$,  
 \[
    B \coloneq \begin{bmatrix}
        0 & 0 & 0 & \dots & 0 \\
        1 & 0 & 0 & \dots & 0 \\
        0 & 2 & 0 & \dots & 0 \\
        \vdots & \ddots & \ddots & \ddots & \vdots \\
        0 & \cdots & 0 & n-1 & 0
    \end{bmatrix}
    \quad \text{and} \quad
    C \coloneq \begin{bmatrix}
        0 & n-1  & 0 & \dots & 0 \\
        0 & 0 & n-2 & \dots & 0 \\
          \vdots & \ddots & \ddots & \ddots & \vdots \\
        0 & 0 & 0 & \dots & 1 \\
        0 & 0  & 0 &\cdots  & 0
    \end{bmatrix}.
\]
 Then $\mathfrak{g}_X = \langle D \rangle$, unless one of the following holds:
\begin{itemize}
    \item $(e_1, \ldots, e_n) = (0, 1, \ldots, n-1)$, in which case 
          $\mathfrak{g}_X = \langle D, B \rangle$, or
    \item $(e_1, \ldots, e_n) = (1-n, 2-n, \ldots, 0)$, in which case 
          $\mathfrak{g}_X = \langle D, C \rangle$.
\end{itemize}
\end{lm}

\begin{proof}
We use \Cref{cor: symmetry via jacobian}. Let 
$\phi' = (e_1 t^{e_1-1},\ldots,e_n t^{e_n-1})^T$ 
be the vector of derivatives of~$\phi$. Observe that 
$D.\phi = t\phi'$, so $D \in \frakg_X$ for any choice of 
$e_1,\ldots,e_n$.

Let $(e_1,\ldots,e_n) \neq (0,\ldots,n-1),(-n+1,\ldots,0)$. 
Any $A\in\frakg_X$ satisfies $A.\phi = c(t)\phi'$. The $i$-th row reads:
$
\sum_j A_{ij}\,t^{e_j} = c(t)e_i\,t^{e_i-1}.
$
The left side is supported on $\{t^{e_j}\}$, so $c(t)$ must be
supported on $\{t^{e_j-e_i+1}\}$ for every $i$ with $e_i\neq 0$.
Since $(e_1,\ldots,e_n)$ is not a consecutive integer sequence,
this intersection over all such $i$ contains exactly one monomial,
forcing $c(t) = \gamma t^m$. Substituting back, $A_{ij}\neq 0$
requires $e_j = e_i+(m-1)$. But $m>1$ gives
$e_n+(m-1)\notin\{e_1,\ldots,e_n\}$ and $m<1$ gives
$e_1+(m-1)\notin\{e_1,\ldots,e_n\}$, so $m=1$ and
$A = \gamma D$.

If $(e_1 ,\ldots, e_n) = (0 ,\ldots, n-1)$ then $B.\phi = \phi'$ showing that $B \in \frakg_X$ as well. We will show that $ \langle B,D \rangle= \frakg_X$. Now,  a  matrix $A\in \frakg_X$ satisfies $A.\phi=c(t)\phi'=c(t)B.\phi$. Extending this expression, since the entries
of $A.\phi$ are polynomials of degree at most $n-1$, comparing
degrees forces $\deg(c) \leq 1$, so $c(t) = \alpha + \beta t$. So, 
\[A.\phi = \alpha B.\phi + \beta t B.\phi
       = \alpha B.\phi + \beta D.\phi
       = (\alpha B + \beta D).\phi,
\]
where we used $tB.\phi = t\phi' = D.\phi$.
 Since the vectors
$\phi(t)$ span $\CC^n$, we conclude $A = \alpha B + \beta D$,

Finally, suppose $(e_1, \ldots, e_n) = (1-n, 2-n, \ldots, 0)$. Consider the composition of the substitution $t \mapsto t^{-1}$ and the coordinate change $x_i \mapsto x_{n+1-i}$. This sends $\phi(t)$ to  the $(0,1,\ldots,n-1)$ case. Under this transformation, $ \diag(1-n,2-n, \ldots,  0) \mapsto -\diag(0,1 \ldots, n-1)$ and $B \mapsto C$, concluding that~$\mathfrak{g}_X = \langle D, C \rangle$.
\end{proof}
}}

The situation analyzed in \Cref{lem:monomial curves} is the general 
situation whenever $\frakg_X\neq \{0\}$.
{{
\begin{thm}\label{thm: Curves with nonzero symmetry}
Fix $n \geq 2$, let $\phi = (f_1,\ldots,f_n)\colon \CC \dto \CC^n$ 
be a rational map and  $X =\overline{\im(\phi)}$. Assume that 
the linear span of $X$ is all of $\CC^n$. 
Then $\frakg_X\neq \{0\}$ if 
and only if, after a linear change of coordinates on $\CC^n$, we have 
$f_i = r^{e_i}$ for some non-constant rational function $r$ and integer 
exponents $e_1 < e_2 < \cdots < e_n$ with $\gcd(e_1,\ldots,e_n)=1$. 
Moreover, $\frakg_X$ is one-dimensional, unless $(e_1,\ldots,e_n)$ is 
$(0,\ldots,n-1)$ or $(-n+1,\ldots,0)$, in which case $\frakg_X$~is~two-dimensional.
\end{thm}

\begin{proof}
The $(\Leftarrow)$ direction follows from \Cref{lem:monomial curves}: 
if $f_i = r^{e_i}$ then $D = \diag(e_1,\ldots,e_n) \in \frakg_X$ is 
nonzero. 

We now prove $(\Rightarrow)$. The symmetry Lie algebra $\frakg_X$ is the Lie algebra of the algebraic 
group~$G_X$. We split the proof into two cases according to whether $\frakg_X$ 
contains a diagonalizable element or not.

\medskip
\noindent\textit{Case 1: $\frakg_X$ contains a diagonalizable element.}
Then $G_X$ contains a one-dimensional torus $T$, and up to a linear 
change of coordinates,
\(
T = \{\diag(t^{e_1},\ldots,t^{e_n}) \mid t \in \CC^*\}
\)
for suitable integers $e_1 \leq \cdots \leq e_n$ with 
$\gcd(e_1,\ldots,e_n)=1$. Now $X$ is an orbit closure of $T$, and 
after scaling coordinates we have 
\[
X = \overline{T\cdot\mathbf{1}} = 
\overline{\{(t^{e_1},\ldots,t^{e_n})\mid t\in\CC^*\}}.
\]
Since $X$ spans $\CC^n$, we have $e_1 < \cdots < e_n$. Note that 
$\prod_i f_i^{a_i}=1$ for all $a_1,\ldots,a_n\in\ZZ$, satisfying 
$\sum_i a_i e_i=0$. 
Choose $b_1,\ldots,b_n\in\ZZ$ with 
$\sum_{i=1}^n b_i e_i=1$ and define \[r=\prod_{i=1}^n f_i^{b_i}.\]
Then $f_j = r^{e_j}$ for $j=1,\ldots, n$. Indeed,
\[
\dfrac{f_j}{r^{e_j}} = f_j\Bigl(\prod_{i=1}^n f_i^{-b_i}\Bigr)^{e_j} 
= \prod_{i=1}^n f_i^{\delta_{ij}-e_j b_i} = 1,
\]
since  $\sum_{i=1}^n (\delta_{ij}-e_j b_i)e_i = e_j - e_j\sum_{i=1}^n b_i e_i = 0$. 

\medskip
\noindent\textit{Case 2: $\frakg_X$ contains no diagonalizable element.}
We show this leads to a contradiction. In this case $G_X^0$ is a 
unipotent algebraic group, so $\frakg_X$ consists entirely of nilpotent 
matrices. We claim $\frakg_X$ contains a nilpotent matrix of rank $n-1$.

After a linear change of coordinates, assume $d_i \coloneq \deg(f_i)$ 
satisfy $d_1 < \cdots < d_n$. By \Cref{cor: symmetry via jacobian}, 
any nonzero $A\in\frakg_X$ satisfies
\[
A.\phi(t) = c(t)\phi'(t) = c(t)(f_1'(t),\ldots,f_n'(t))^T
\]
for some nonzero $c(t)\in\CC(t)$ (nonzero since $X$ spans $\CC^n$). 
Among the $d_i$, at most one equals zero, so at most one entry of 
$\phi'(t)$ is identically zero, and the remaining entries have distinct 
degrees. Hence $c(t)\phi'(t)$ has image of dimension $n-1$, so $A$ 
has rank $n-1$, as claimed.
Since nilpotent matrices of rank $n-1$ form a single orbit under conjugation with
 $\GL_n$, after another linear change of coordinates we may assume 
$A = B$ from \Cref{lem:monomial curves}. Then $f_1'\equiv 0$, so 
$f_1 = 1$ after scaling; and $f_i'(t) = (i-1)c(t)f_{i-1}(t)$ for 
$i=2,\ldots,n$. Setting $c(t)=f_2'(t)$ and arguing by induction, each 
$f_i(t)$ is a linear combination of $1, f_2(t),\ldots,f_2^{i-1}(t)$ 
with leading coefficient $1$. After a further linear change of 
coordinates, $f_i = f_2^{i-1}$ for all $i$. After reparameterizing,
we arrive at the case $(e_1,\ldots,e_n)=(0,\ldots,n-1)$. But  
$D = \diag(0,1,\ldots,n-1)\in\frakg_X$ is diagonalizable from \Cref{lem:monomial curves}, 
contradicting our assumption.
\end{proof}
}}

In the case $(e_1,\ldots,e_n) = (0,\ldots,n-1)$, consider the affine 
space $\{y_1 = 1\} \subseteq \CC^n$ as an affine chart of $\PP^{n-1}$. 
The projective closure $\overline{X} \subseteq \PP^{n-1}$ is then a 
rational normal curve, and its intersection with the hyperplane at 
infinity is a single point $q$. In this case, the symmetry Lie algebra of $\overline{X} \subseteq \PP^{n-1}$ is a copy of $\fraksl_2$, and the two-dimensional algebra $\frakg_X = \langle D, B\rangle$ of \Cref{thm: Curves with nonzero symmetry} is the Borel subalgebra~of~$\fraksl_2$ stabilizing $q$.

More generally, when $(e_1,\ldots,e_n) = (e_1, e_1+1,\ldots,e_1+n-1)$ with $-n+2 \leq e_1 \leq -1$, the curve $X$ is the complement of a hyperplane section of a rational normal curve $\overline{X}$. In this case, the hyperplane section consists of two points $\overline{X}$, and $\frakg_X$ is the one-dimensional subalgebra of~$\fraksl_2$ stabilizing both points, namely the Cartan subalgebra spanned by $D$.

{{We illustrate \Cref{thm: Curves with nonzero symmetry} in the case of the binary one-stage tree models from \Cref{sec:computational experiments}. 
\begin{ex}\label{ex: staged tree curve}
Consider the curve $X$ parametrized by
   \begin{figure}[ht]
       \centering  \includegraphics[width=0.4\linewidth]{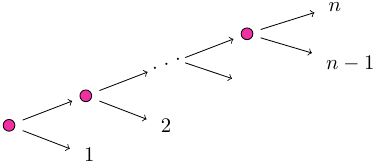}
       \caption{A stage tree model realized by the rational normal curve of degree $n-1$.}
  \label{fig:binaryonestage}
   \end{figure}
    \[
    \phi(\theta) = ( 1-\theta, \theta(1-\theta), \theta^2(1-\theta) ,\ldots, \theta^{n-2}(1-\theta), \theta^{n-1}) \in \CC^{n}.
    \]
This is the parametrization for the binary staged tree model in \Cref{fig:binaryonestage}. 

The matrices
\begin{align*}\label{eq: staged tree basis for g_X}
\small
D' =
\begin{bmatrix}
0 & -1 & -1 & \cdots & -1 & -1 \\
0 & 1 & -1 & \cdots & -1 & -1 \\
0 & 0 & 2 & \cdots & -1 & -1 \\
0 & 0 & 0 & \ddots & \vdots & \vdots \\
0 & 0 & 0 & \cdots & n-2 & -1 \\
0 & 0 & 0 & \cdots & 0 & n-1
\end{bmatrix} \ \text{ and }
\ 
B' = 
\begin{bmatrix}
-1 & -1 & -1 & \cdots & -1 & -1 \\
1 & -1 & -1 & \cdots & -1 & -1 \\
0 & 2 & -1 & \cdots & -1 & -1 \\
0 & 0 & 3 & \cdots & -1 & -1 \\
\vdots & \vdots & \vdots & \ddots & \vdots & \vdots \\
0 & 0 & 0 & \cdots & n-1 & -1
\end{bmatrix}
\end{align*}
 are in $\frakg_X$ as they satisfy \Cref{cor: symmetry via jacobian}. Now, \Cref{thm: Curves with nonzero symmetry} guarantees that $\dim \frakg_X = 2$ and that $X$ is the rational normal curve of degree $n-1$. After the change of coordinates given by 
\[
g=\begin{bmatrix}
1 & -1 & 0 & \cdots & 0 & 0 \\
0 & 1 & -1 & \cdots & 0 & 0 \\
0 & 0 & 1 & \cdots & 0 & 0 \\
\vdots & \vdots & \vdots & \ddots & \vdots & \vdots \\
0 & 0 & 0 & \cdots & 1 & -1 \\
0 & 0 & 0 & \cdots & 0 & 1
\end{bmatrix}
\]
we have $g^{-1} D'g = D$ and $g^{-1}B'g = B$ where $D$ and $B$ are the matrices of \Cref {lem:monomial curves} and $g^{-1} \phi(\theta) = (1, \theta ,\ldots, \theta^{n-1})$. 
\end{ex}
}}

\section{Secant varieties} \label{sec:Secants}

Let $X \subseteq \CC^n$ be an algebraic variety invariant under rescaling, that is if $q \in X$ then $\lambda q \in X$ for all $\lambda \in \CC$. Note that in this case $T_q X = T_{\lambda q} X$ for all $\lambda \neq 0$. The variety
\[
\sigma_r (X) = \overline{\{ q_1 + \cdots + q_r \mid q_i \in X\}}
\]
is called the $r$-th secant variety of $X$. It is immediate that $G_X \subseteq G_{\sigma_r(X)}$. There are examples for which the inclusion is strict, see, e.g., \cite[Prop. 4.1]{GHL}. In this section, we provide a sufficient condition for which, instead, equality holds at least at the level of the Lie algebra. Note that a standard parameter count provides 
\[
\dim \sigma_r(X) \leq r \dim X ;
\]
we say that $X$ has the \emph{expected dimension} if equality holds. Terracini's Lemma describes the tangent space to $\sigma_r(X)$ at a point $z = q_1 + \cdots + q_r$ with $q_i \in X$ sufficiently general as:
\[
T_z \sigma_r(X) = T_{q_1} X + \cdots + T_{q_r} X.
\]
Note that $\sigma_r(X)$ has the expected dimension if and only if the sum above is a direct sum. In particular, if $\sigma_r(X)$ has the expected dimension, the same holds for  $\sigma_s(X)$ if $s \leq r$.
\begin{thm}\label{thm: secant varieties lie algebra}
Let $X$ be an algebraic variety invariant under rescaling such that $\sigma_{r+1}(X)$ has the expected dimension. Then 
\(
\frakg_{\sigma_r(X)} = \frakg_X.
\)
\end{thm}
\begin{proof}
We proceed by induction on $r$. For $r = 1$, there is nothing to prove. Let $r \geq 2$ and assume $\frakg_X = \frakg_{\sigma_{r-1}(X)}$; we will prove that $\frakg_X = \frakg_{\sigma_r(X)}$. The inclusion $\frakg_X \subseteq \frakg_{\sigma_r(X)}$ is clear. To prove the reverse inclusion, let $A \in \frakg_{\sigma_r(X)}$. By \Cref{thm: frakg via tangent spaces}, for any sufficiently general~$z \in \sigma_r(X)$, we have $A.z \in T_z \sigma_r(X)$.

Let $q_1 ,\ldots, q_{r-1} ,q_{r},q_{r}' \in X$ be elements chosen in a suitable Zariski open set and let 
\begin{align*}
y &= q_1 + \cdots + q_{r-1}, \\
z &= q_1 + \cdots + q_{r-1} + q_r, \\
z' &= q_1 + \cdots + q_{r-1} + q_{r}', \\
w &= q_1 + \cdots + q_{r-1} + q_r + q_r'.
\end{align*}
By Terracini's Lemma, we have
\begin{align*}
T_y \sigma_{r-1} (X) &= T_{q_1} X + \cdots + T_{q_{r-1}} X, \\
T_z \sigma_r(X) &= T_{q_1} X + \cdots + T_{q_{r-1}} X + T_{q_r} X, \\ 
T_{z'} \sigma_r(X) &= T_{q_1} X + \cdots + T_{q_{r-1}} X + T_{q_{r}'} X, \\
T_w \sigma_{r+1}(X) &= T_{q_1} X + \cdots + T_{q_{r-1}} X + T_{q_r} X + T_{q_r'} X.
\end{align*}
Since $\sigma_{r+1}(X)$ has the expected dimension, these are direct sums, showing that 
\[
T_y \sigma_{r-1}(X) = T_z \sigma_r(X) \cap T_{z'} \sigma_r(X).
\]
By linearity of $A$, and the fact that $T_{\lambda q_r} X = T_{q_r} X$, we deduce that $A.y \in T_y \sigma_{r-1}(X)$. Indeed, for almost every $\lambda$, $z_\lambda = q_1 + \cdots + q_{r-1} + \lambda q_r$ satisfies
\[
A.z_\lambda \in T_{z_\lambda} \sigma_r(X) = T_{z} \sigma_r(X);
\]
by continuity and by linearity, since $y = z_\lambda|_{\lambda = 0}$, we have $A.y \in T_{z} \sigma_r(X)$. Similarly $A.y \in T_{z'} \sigma_r(X)$. We deduce $A.y \in T_{z} \sigma_r(X) \cap T_{z'} \sigma_r(X) = T_y \sigma_{r-1}(X)$. This shows $A \in \frakg_{\sigma_{r-1}(X)}$. By the induction hypothesis $\frakg_{\sigma_{r-1}(X)} = \frakg_X$,  which implies $A \in \frakg_X$ as desired.
\end{proof}

The study of dimension of secant varieties is a classical theme in algebraic geometry. Extensive research is done especially on the case of Segre-Veronese varieties. Let $V_1 ,\ldots, V_k$ be vector spaces and let $d_1 ,\ldots, d_k$ be positive integers. Let $S^{d_i}V_i \subseteq V_i^{\otimes d_i}$ be the $d_i$-th symmetric power of $V_i$, that is the space of tensors which are invariant under the action of $\frakS_{d_i}$ on $V_i^{\otimes d_i}$ which permutes the tensor factors. The Segre-Veronese variety is 
\[
X = \{v_1^{\otimes d_1} \ootimes v_k ^{\otimes d_k}  \mid v_i \in V_i \} \subseteq S^{d_1} V_1 \otimes \cdots \otimes S^{d_k} V_k.
\]
When $k =1$, $X$ is the Veronese variety, whose secant varieties have the expected dimension except for a small number of classical exceptions \cite{AlHir}; the classification of secant varieties having dimension smaller than expected is almost complete for the case of Segre varieties, that is when $d_1 = \cdots = d_k = 1$ \cite{BocChiOtt}; partial results in several regimes are known for the general setting \cite{AboBra,ABGO24,DK25}. 

The symmetry Lie group of Segre-Veronese varieties in $\GL( S^{d_1} V_1 \ootimes S^{d_k} V_k)$ is the image of the subgroup $(\GL(V_1) \ttimes \GL(V_k)) \rtimes \frakS$ where $\GL(V_i)$ acts on the factor $S^{d_i}V_i$ and $\frakS \subseteq \frakS_k$ is the subgroup that permutes the factors of the same dimension that have the same degree \cite[Prop. 2.2]{GHL}. The following result is a direct application of \Cref{thm: secant varieties lie algebra}.

\begin{cor}\label{cor: secant varieties}
Let $X$ be a Segre-Veronese variety such that $\sigma_{r+1}(X)$ has the expected dimension. Then
\[
G_{\sigma_r(X)} = G_X \cong (\GL(V_1) \ttimes \GL(V_k))/(\CC^{\times})^{k-1} \rtimes \frakS,
\]
where $(\CC^{\times})^{k-1}$ is the kernel of the action.
\end{cor}
\begin{proof}
The assumption on the dimension, together with \Cref{thm: secant varieties lie algebra}, guarantees~that~$\frakg_{\sigma_r(X)} = \frakg_X$ and therefore $G^0 _{\sigma_r(X)} = G^0_ X$ is the image of $\GL(V_1) \ttimes \GL(V_k)$ in $\GL( S^{d_1} V_1 \ootimes S^{d_k} V_k)$. Since the identity component is a normal subgroup of the symmetry group, we have $G_{\sigma_r(X)} \subseteq N_{\GL( S^{d_1} V_1 \ootimes S^{d_k} V_k)} ( G_{\sigma_r(X)}^{\circ})$. This normalizer is the whole $ G_X $, see \cite[Prop. 3.3]{GHL}.
\end{proof}

We point out that in many cases the symmetry Lie algebra of an $r$-th secant variety coincides with that of the base variety, even if the $(r+1)$-th (and often even the $r$-th) secant variety has dimension smaller than expected. A classical example are  secant varieties of Segre varieties with two factors, that is, varieties of matrices of bounded rank, which have the expected symmetry Lie algebra whenever they do not fill the ambient space \cite{GUTERMAN200061}. The Segre varieties of $3$ factors of dimension $3$, and of $4$ factors of dimension $2$ have defective secant variety for~$r = 4$ and~$r=3$ respectively; however, the results of \cite[Sec. 4, Sec. 9]{GHL} show that all secant varieties of these Segre varieties have the expected symmetry group. In fact, the only examples that we know for which a secant variety of a Segre-Veronese variety has symmetry Lie algebra larger than the one of the base variety are pencils of matrices \cite[Prop. 4.1]{GHL}.

{
\bibliographystyle{alphaurl}
\bibliography{sym.bib}
}

\end{document}